\documentclass[11pt,oneside,reqno]{amsart}

\usepackage[english]{babel}
\usepackage{amsmath,amssymb,amsthm}

\theoremstyle{plain}
\newtheorem{theorem}{Theorem}
\newtheorem{lemma}{Lemma}
\newtheorem{proposition}{Proposition}
\newtheorem{corollary}{Corollary}

\theoremstyle{definition}
\newtheorem{definition}{Definition}
\newtheorem{example}{Example}

\theoremstyle{remark}
\newtheorem{remark}{Remark}

\usepackage[page,toc,titletoc,title]{appendix}
\usepackage{graphicx}
\usepackage{xcolor}
\usepackage{hyperref}

\DeclareRobustCommand{\gobblefive}[5]{}
\newcommand*{\SkipTocEntry}{\addtocontents{toc}{\gobblefive}}

\definecolor{darktaupe}{rgb}{0.24, 0.08, 0.08}
\definecolor{blueN}{rgb}{0.0, 0.58, 0.79}
\hypersetup{colorlinks=true,
            breaklinks,
            linkcolor=blueN,
            citecolor=green,
            urlcolor=blueN}

\usepackage[nameinlink, noabbrev, capitalise]{cleveref}
\usepackage{algorithm,algorithmic,matlab-prettifier}
\usepackage{enumitem}
\usepackage[font=small,skip=2pt]{caption,subcaption}
\usepackage{multirow,enumitem}
\usepackage[group-digits=false]{siunitx}
\usepackage{booktabs}

\newcommand{\N}{\mathbb N}
\newcommand{\R}{\mathbb R}
\newcommand{\cA}{\mathcal A}
\newcommand{\cF}{\mathcal F}
\newcommand{\cH}{\mathcal H}
\newcommand{\cL}{\mathcal L}
\newcommand{\eps}{\varepsilon}

\title[Gripenberg-like algorithm for the lower spectral radius]{Gripenberg-like algorithm for the lower spectral radius}

\author[N.~Guglielmi, F.~P.~Maiale]{Nicola Guglielmi, Francesco Paolo Maiale}

\address {Nicola Guglielmi\newline \indent
	Gran Sasso Science Institute \newline \indent
	Viale F. Crispi, 7, 67100 L'Aquila, Italy}
\email{nicola.guglielmi@gssi.it}

\address {Francesco Paolo Maiale \newline \indent
	Gran Sasso Science Institute \newline \indent
	Viale F. Crispi, 7, 67100 L'Aquila, Italy}
\email{francescopaolo.maiale@gssi.it}

\begin{document}


\begin{abstract}
This article presents an extended algorithm for computing the lower spectral radius of finite, non-negative matrix sets. Given a set of matrices $\mathcal{F} = \{A_1, \ldots, A_m\}$, the lower spectral radius represents the minimal growth rate of sequences in the product semigroup generated by $\mathcal{F}$. This quantity is crucial for characterizing optimal stable trajectories in discrete dynamical systems of the form $x_{k+1} = A_{i_k} x_k$, where $A_{i_k} \in \mathcal{F}$ for all $k \ge 0$.
For the well-known joint spectral radius (which represents the highest growth rate), a famous algorithm providing suitable lower and upper bounds and able to approximate the joint spectral radius with arbitrary accuracy was proposed by Gripenberg in 1996. For the lower spectral radius, where a lower bound is not directly available (contrarily to the joint spectral radius), this computation appears more challenging.

Our work extends Gripenberg's approach to the lower spectral radius computation for non-negative matrix families. The proposed algorithm employs a time-varying antinorm and demonstrates rapid convergence. Its success is related to the property that the lower spectral radius can be obtained as a Gelfand limit, which was recently proved in Guglielmi and Zennaro (2020). Additionally, we propose an improvement to the classical Gripenberg algorithm for approximating the joint spectral radius of arbitrary matrix sets.
\end{abstract}

\maketitle
\setcounter{tocdepth}{1}
\tableofcontents

%%%%%%%%%%%%%%%%%%%%%%%%%%%%%%%%%%%%%%%%%%%%%%%%%%%%%%%%%%%%%%%%%%%%%%%%%%%%%%%%%

\section{Introduction} \label{sec:introduction}

The most important joint spectral characteristics of a set of linear operators are the joint spectral radius and the lower spectral radius. The first one is largely studied while the second one, even if not less important, is less investigated, mainly due to the difficulties connected with its analysis and computation. The joint spectral radius (JSR) of a set of matrices represents the highest possible rate of growth of the products generated by the set. For a single matrix $A$ the JSR is clearly identified by the spectral radius $\rho(A)$. However, for a set of matrices, such a simple characterization is not available, and one has to analyze the product semigroup, where the products have to be considered with no ordering restriction and allowing repetitions of the matrices. The lower spectral radius (LSR) aka subradius of a set of matrices defines instead the lowest possible rate of growth of the products generated by the set. For a single matrix $A$ the LSR still coincides with its spectral radius $\rho(A)$, but for a set of matrices its computation (or approximation) turns out to be even more difficult than for the JSR.

The LSR is a very important stability measure, since it is related to the stabilizability of a dynamical system. In control theory (see, e.g., \cite{JuMa17}) its computation identifies the most stable obtainable trajectories, whose importance is often crucial in terms of the behavior of the controlled system. Moreover, it allows to compute, for example, the lower and upper growth of the Euler partition functions \cite{Eul45,Kru50}, which is relevant to the theory of subdivision schemes and refinement equations - refer to \cite{Pro17} for more details.

The JSR, on the other hand, appears to be an important measure in many applications including discrete switched systems (see, e.g., Shorten et. al.~\cite{MaSh07}), convergence analysis of subdivison schemes, refinement equations and wavelets (see, e.g., Daubechies and Lagarias~\cite{DauLag91,DauLag92}, \cite{MoRe14}), numerical stability analysis of ordinary differential equations (see, e.g., Guglielmi and Zennaro~\cite{GugZen00}), regularity analysis of subdivision schemes \cite{GugPro16}, as well as coding theory and combinatorics, for which we refer to the survey monography by Jungers~\cite{J2009}.

The JSR is largely studied in the literature and there exist several algorithms aiming to compute it. The LSR, instead, is less studied and turns out to be more difficult to compute or even approximate.

In this article, we mainly direct our attention to the approximation of the LSR, but we also consider the JSR. Specifically, we discuss several algorithms, which we call Gripenberg-like, as they extend the seminal work of Gripenberg \cite{Gripenberg1996}. These algorithms allow to estimate, within a specified precision $\delta$, both quantities even in relatively high dimension.

\subsection{Main results of the article} 

The algorithms for computing the LSR and JSR are based on the identification of suitable lower and upper bounds. For any product $P$ of degree $k$ in the product semigroup, $\rho(P)^{1/k}$ provides an upper bound for the LSR and a lower bound for the JSR. Moreover, any matrix norm  of the considered family $\cF$, $\| \cF \| = \max_{1 \le i \le m} \| A_i \|$, provides an upper bound for the JSR. For the LSR the role of the norm is played instead by the so-called antinorm, which - however - can be only defined in a cone. Antinorms are continuous, nonnegative, positively homogeneous and superadditive functions defined on the cone, and turn out to provide a lower bound to the LSR.

The polytope algorithm, introduced in \cite{GP13}, is a commonly used method for computing the LSR of a family of matrices sharing an invariant cone. It involves two key steps, which consist respectively
\begin{enumerate}[label=(\roman*)]
	\item of the identification of a \textit{spectrum lowest product} (s.l.p.), that is a product $P$ of degree $k$ in the multiplicative semigroup, with $\rho(P)^{1/k}$ minimal with respect to all other products;
	\item of the successive construction of a special polytope antinorm for the family of matrices (refer to \cref{sec:invcone_antinorm} and \cref{def:pol_antinorm}). 
\end{enumerate}
For a product $P$ of degree $k$, $\rho(P)^{1/k}$ provides an upper bound for the LSR, while any antinorm of the family $\cF$ provides a lower bound. If these lower and upper bounds coincide, the LSR is computed exactly.

The main contribution of this paper is a Gripenberg-like algorithm (for the classical Gripenberg's algorithm for the JSR, see \cite{Gripenberg1996}) that estimates the LSR within a given precision $\delta>0$ for a specific class of matrix families (finite and satisfying a certain assumption - \cref{def:asympt_rank_one} -, which appears to be generic).

While the existing literature offers robust methods for computing the LSR, such as the polytopic algorithm \cite{GP13}, our approach differs in both aim and methodology. Specifically, unlike the polytopic algorithm, which requires identifying a s.l.p. and then constructing an optimal extremal polytope antinorm, our algorithm:
\begin{itemize}
	\item guarantees convergence from any initial polytope antinorm;
	\item aims to compute the smallest possible set of products to estimate the LSR, thus reducing computational cost.
\end{itemize}

We also introduce an adaptive procedure (based on the polytopic algorithm) to iteratively refine the initial antinorm, significantly improving both the rate and accuracy of convergence. This variant, discussed in \cref{subsec:algorithm_enhanced}, performs better in our simulations, especially in high-dimensional cases, as shown in \cref{sec:random}.

\subsection{Outline of the article} 

In \cref{sec:preliminaries}, we discuss all the preliminary material necessary for computing the lower spectral radius.

In \cref{sec:algorithm}, we present the Gripenberg-like algorithm for the computation of the LSR and in \cref{subsec:proof_main} we prove our main result. Subsequently, we discuss theoretical improvements: an adaptive procedure (\cref{subsec:algorithm_enhanced}) and a technique based on perturbation theory (\cref{subsec:perturbations}) to extend the considered class of problems.

In \cref{sec:numerical_implementation}, we detail the numerical implementation of our algorithm and its adaptive variants. Further considerations on these adaptive variants, as well as antinorm calculations, are provided in the appendix.

In \cref{sec:il_ex}, we present two illustrative examples that demonstrate both the advantages and critical aspects of our algorithm. In addition, we propose alternative strategies and analyze the corresponding improvements in LSR computation.

In \cref{sec:numerical_applications}, we discuss some applications to number theory and probability, that have already been investigated in \cite{GP13}, for a comparative analysis.

In \cref{sec:random}, we present data collected from applying our algorithms to randomly generated matrix families both full and sparse with varying sparsity densities.

In \cref{sec:conclusion}, we provide concluding remarks and briefly discuss an improved version of Gripenberg's algorithm for the JSR, which is described in \cref{subsec:adaptive_norm_algorithm}.

\subsection*{Outline of the appendix} 

The appendix provides supplementary information on polytope antinorm calculations, numerical implementations of adaptive algorithms, and an illustrative output. It is structured as follows:

\begin{itemize}
	\item \Cref{supplement_sec:comp_antinorm} discusses an efficient numerical implementation for computing \textbf{polytope antinorms}, while \cref{supplementsec:pruning_vertices} outlines the vertex set pruning procedure required in adaptive algorithms.
	
	\item The numerical implementation of adaptive algorithms, Algorithm \textbf{(A)} and Algorithm \textbf{(E)}, is described and commented in \cref{supplement_sec:adaptive_alg}.
	
	\item \Cref{supplement:simulation} provides a detailed numerical output from the first two steps of Algorithm \textbf{(A)} applied to the illustrative example introduced in \cref{subsec:il_ex}.
	
\end{itemize}

%%%%%%%%%%%%%%%%%%%%%%%%%%%%%%%%%%%%%%%%%%%%%%%%%%%%%%%%%%%%%%%%%%%%%%%%%%%%%%%%%
\section{Lower spectral radius}
\label{sec:preliminaries}

Consider a finite family $\cF = \{A_1, \ldots, A_m\}$ of $d \times d$ real-valued matrices. For each $k \in \N$, we define the set $\Sigma_k(\cF)$, containing all possible products of degree $k$ generated from the matrices in $\cF$, as follows:
\begin{equation} \label{eq:Sigmak}
	\Sigma_k(\cF) := \left\{ \ \prod_{j=1}^k A_{i_j} \: : \: i_j \in \{1,\ldots,m\} \text{ for every } j = 1,\ldots,k \right\}.
\end{equation}
Given a norm $\| \cdot \|$ on $\R^d$ and the corresponding induced matrix norm, the \textit{lower spectral radius} (LSR) of $\cF$ is defined as (see \cite{Gurvits1995}):
\begin{equation}\label{eq:lsr_def}
	\check{\rho}(\cF) := \lim_{k\to+\infty} \min_{P \in \Sigma_k(\cF)} \| P \|^{1/k}.
\end{equation}
On the other hand, if we set
\[
\bar{\rho}_k(\cF) := \min_{P \in \Sigma_k(\cF)} \rho(P)^{1/k},
\]
then (see, for example, \cite{BochiMorris2015} and \cite{Czornik2005}) we have the following equality:
\[
\check{\rho}(\cF) = \inf_{k \ge 1} \bar{\rho}_k(\cF).
\]
Thus, the spectral radius of a product provides an upper bound for the LSR in our algorithm. However, we lack a corresponding lower bound. To address this, we assume that $\cF$ shares an invariant cone $K$ (e.g., $\R_+^d$). Under this assumption, we introduce in \cref{subsec:antinorms_prop_ex} the notion of an \textit{antinorm} defined on a convex cone $K \subset \R^d$, which will be the tool to obtain a lower bound (in the same way as norms provide upper bounds to the JSR). We now introduce an important quantity for our analysis:

\begin{definition} \label{def:spm}
	A product $\Pi \in \Sigma_k$ is a \textbf{spectral lowest product} (s.l.p.) if
	\[
	\check \rho(\cF) = \rho(\Pi)^{1/k}.
	\]
\end{definition}

The lower spectral radius determines the asymptotic growth rate of the minimal product of matrices from $\cF$. This notion, first defined in \cite{Gurvits1995}, has diverse applications across mathematics, including combinatorics, and number theory. For a comprehensive overview, see \cite{J2009} and the references therein.

%%%%%%%%%%%%%%%%%%%%%%%%%%%%%%%%%%%%%%%%%%%%%%%%%%%%%%%%%%%%%%%%%%%%%%%%%%%%%%%%%
\subsection{Invariant cones} \label{sec:invcone_antinorm}

In this section, we recall the notion of proper cone, a well-established topic in the literature (see, for example, the papers \cite{Rodman2010, Schneider2014,Tam2001}).

\begin{definition}
	A proper cone $K$ of $\R^d$ is a nonempty closed and convex set which is:
	\begin{enumerate}[label=(\arabic*)]
		\item positively homogeneous, i.e. $\R_+ K \subseteq K$;
		\item salient, i.e. $K \cap -K = \{0\}$;
		\item solid, i.e. $\mathrm{span}(K) = \R^d$.
	\end{enumerate}
	In addition, the dual cone is defined as 
	\[
	K^\ast := \left\{ y \in \R^d \: : \: x^T y \ge 0 \text{ for all } x \in K \right\}.
	\]
\end{definition}

\begin{definition}
	Let $A \in \R^{d,d}$. A proper cone $K$ is \textbf{invariant} for $A$ if 
	\[
	A(K) \subseteq K,
	\]
	and \textbf{strictly invariant} if 
	\[
	A\left(K\setminus \{0\}\right) \subseteq \mathrm{int}(K).
	\]
	In particular, $K$ is (strictly) invariant for the family $\cF = \{A_1,\ldots,A_m\}$ if it is (strictly) invariant for each matrix $A_i \in \cF$.
\end{definition}

\begin{proposition}[\cite{GZ21}]\label{prop:inv_cone}
	A cone $K$ is (strictly) invariant for $\cF$ if and only if the dual $K^\ast$ is (striclty) invariant for $\cF^T = \{A_1^T,\ldots,A_m^T\}$.
\end{proposition}

\begin{remark} \label{rmk:self_dual_cone}
	The cone $K = \R_+^d$ is self-dual. Thus, it is strictly invariant for $\cF$ if and only if it is strictly invariant for $\cF^T$.
\end{remark}

We now turn our attention to \textit{asymptotically rank-one matrices}. First, we briefly recall the spectral decomposition, adopting the notation used in \cite{BrunduZennaro2018,BrunduZennaro2019AMP,BrunduZennaro2019JCA}:

Let $\lambda \in \mathbb C$ be an eigenvalue of $A$ with algebraic multiplicity $k$. Denote by $\widetilde V_\lambda$ and $\widetilde W_\lambda$ respectively the associated eigenspace and the generalized eigenspace:
\[
\widetilde W_\lambda := \operatorname{Ker} \left(( A - \lambda I)^k \right) \supseteq  \operatorname{Ker} \left( A - \lambda I \right) =: \widetilde V_\lambda.
\]
If $\lambda \in \mathbb R$, then $W_\lambda := \widetilde W_\lambda \cap \R^d$ and $V_\lambda := \widetilde V_\lambda \cap \R^d$ are linear subspaces of $\mathbb R^d$, invariant under the action of $A$. On the other hand, when $\lambda \notin \mathbb R$, we define
\[
U_{\R}(\lambda, \bar \lambda) := \left( \widetilde W_\lambda \oplus \widetilde W_{\bar \lambda} \right) \cap \R^d.
\]
Since $\lambda$ and $\bar \lambda$ are eigenvalues with the same multiplicity, it is easy to check that $U_\R$ is a linear subspace of $\mathbb R^d$ invariant under the action of $A$.

\begin{remark}
	Assume $\rho(A)>0$ and there is a leading eigenvalue $\lambda_1 \in \mathbb R$. Then the spectral theorem yields a decomposition of the space:
	\[
	\R^d = W_A \oplus H_A,
	\]
	where $W_A := W_{\lambda_1}$ is the generalized eigenspace of $\lambda_1$, and $H_A$ is defined as
	\begin{equation}\label{eq:secondary_eig}
		H_A := \left( \bigoplus_{i=2}^r W_{\lambda_i} \right) \oplus \left( \bigoplus_{i=1}^s U_{\mathbb{R}}(\mu_i,\bar \mu_i)\right),
	\end{equation}
	where $\lambda_1,\lambda_2,\ldots,\lambda_r \in \R$ are the real eigenvectors of $A$, while $(\mu_1, \bar\mu_1),\ldots,(\mu_s,\bar\mu_s)$ the complex conjugate pairs.
\end{remark}

\begin{definition}
	$A \in \mathbb R^{d, d}$ is \textbf{asymptotically rank-one} if:
	\begin{enumerate}[label=(\roman*)]
		\item $\rho(A)>0$ and either $\rho(A)$ or $-\rho(A)$ is a simple eigenvalue;
		\item $|\lambda| < \rho(A)$ for any other eigenvalue of $A$.
	\end{enumerate}
\end{definition}

In particular, the leading generalized eigenspace $W_A$ of an asymptotically rank-one matrix is the one-dimensional line $V_{\rho(A)}$.

\begin{definition} \label{def:asympt_rank_one}
	Let $\cF = \{A_1,\ldots,A_m\}$ be a family of real-valued matrices and consider the set of all possible products of any degree:
	\[
	\Sigma(\cF) := \bigcup_{k \ge 1} \Sigma_k(\cF)
	\]
	We say that $\cF$ is \textbf{asymptotically rank-one} if all products $P \in \Sigma(\cF)$ are so.
\end{definition}

\begin{remark}
	It is important to note that even if all matrices $A_1,\ldots,A_m$ are asymptotically rank-one, the same may not be true for their products.
\end{remark}

The following result, which includes the well-known \textit{Perron-Frobenius theorem}, relates invariant cones and asymptotic rank-one matrices:

\begin{theorem}\label{thm.cones}
	If $A \in \R^{d,d}$ and $K$ is an invariant cone for $A$, then:
	\begin{enumerate}[label=(\arabic*)]
		\item The spectral radius $\rho(A)$ is an eigenvalue of $A$.
		\item The cone $K$ contains an eigenvector $v_A$ corresponding to $\rho(A)$.
		\item The intersection $\operatorname{int}(K) \cap H_A$ is empty.
	\end{enumerate}
	If, in addition, $K$ is a strictly invariant cone for $A$, the following holds:
	\begin{enumerate}[label=(\arabic*)]
		\setcounter{enumi}{3}
		\item The matrix $A$ is asymptotically rank-one.
		\item The unique eigenvector of $A$ which belongs to the interior of $K$ is $v_A$.
		\item The intersection $K \cap H_A$ is empty.
	\end{enumerate}
\end{theorem}

For the proof, refer to \cite[Theorem 4.7, Theorem 4.8 and Theorem 4.10]{BrunduZennaro2018}.

%%%%%%%%%%%%%%%%%%%%%%%%%%%%%%%%%%%%%%%%%%%%%%%%%%%%%%%%%%%%%%%%%%%%%%%%%%%%%%%%%
\subsection{Antinorms} \label{subsec:antinorms_prop_ex}

Assume that the family $\cF = \{A_1,\ldots,A_m\}$ shares a common invariant cone $K$. First, we recall the definition of antinorm (\cite{M90}):

\begin{definition}
	An {\itshape antinorm} is a nontrivial continuous function $a(\cdot)$ defined on $K$ satisfying the following properties: \mbox{}
	\begin{enumerate}
		\item[$(a1)$] \textit{non-negativity}: $a(x) \ge 0$ for all $x \in K$;
		\item[$(a2)$] \textit{positive homogeneity}: $a(\lambda x) = \lambda a(x)$ for all $\lambda \ge 0$ and $x \in K$;
		\item[$(a3)$] \textit{superadditivity}: $a(x+y) \ge a(x) + a(y)$ for all $x,y \in K$.
	\end{enumerate}
\end{definition}

\begin{remark}
	By $(a2)$ and $(a3)$, any antinorm $a(\cdot)$ is concave and thus continuous in the interior of the cone $\operatorname{int}(K)$.
\end{remark}

\begin{definition}
	For an antinorm $a(\cdot)$ on $K$, its \textbf{unit antiball} is defined as
	\[
	\cA:= \{x \in K \: : \: a(x) \ge 1\},
	\]
	and its corresponding \textbf{unit antisphere} is $\cA' := \{x \in K \: : \: a(x) = 1\}$.
\end{definition}

\begin{remark}
	Since $a(\cdot)$ is concave, the unit antiball $\cA$ is convex.
\end{remark}

For a given norm $\| \cdot \|$, \cite[Proposition 4.3]{GZ21} proves that even for $a(\cdot)$ discontinuous on $\partial K$, there is $\beta>0$ such that $a(x) \le \beta \| x \|$ for all $x \in K$. However, establishing a lower bound requires an additional assumption on $a(\cdot)$, namely \textit{positivity}:

\begin{definition}
	We say that $a(\cdot)$ is positive if $a(x) > 0$ for all $x \in K \setminus \{0\}$.
\end{definition}

\begin{proposition}
	Let $a(\cdot)$ be a positive antinorm on $K$ and $\| \cdot \|$ be any norm on $\R^d$. Then there are $\beta,\gamma>0$ such that
	\[
	\beta^{-1} a(x) \le \|x\| \le\gamma a(x) \qquad \text{for every } x \in K.
	\]
	In particular, the unit antisphere $\cA'$ is compact.
\end{proposition}

See \cite[Proposition 4.3]{GZ21} for a precise statement and the proof. For finite computability, the following class of antinorms is crucial:

\begin{definition}\label{def:pol_antinorm}
	An antinorm $a(\cdot)$ on a cone $K$ is a \textbf{polytope antinorm} if its unit antiball $\cA$ is a positive infinite polytope. In other words, there exists a set
	\[
	V := \{ v_1, \ldots, v_p \} \subset K \setminus \{0\},
	\]
	minimal and finite of vertices such that:
	\begin{enumerate}[label=(\roman*)]
		\item $a(v_i)=1$ for every $i=1,\ldots,p$;
		\item $\cA= \operatorname{conv}(V)+ K$, where $\operatorname{conv}(V)$ denotes the convex hull of $V$.
	\end{enumerate}
\end{definition}

\begin{remark}
	The term ``polytope'' is used loosely, as the unit antiball $\cA$ might reside in a non-polyhedral cone $K$.
\end{remark}

\subsection*{Examples of antinorms}
Throughout this paper, we focus on matrix families $\cF = \{A_1,\ldots,A_m\}$ with all non-negative entries. This guarantees that $\R_+^d$ is an invariant cone, allowing the use of $p$-antinorms - for more details, refer to \cite{M90}.

\begin{definition}
	Let $K = \mathbb R^d_+$. For $p\le 1$ and $p \neq 0$, we define the $p$-antinorm:
	\[
	a_p(x) := \left(\sum_{i=1}^d x_i^p \right)^{1/p}.
	\]
	Moreover, by letting $p \to - \infty$ we define the antinorm $a_{-\infty}(x) := \min_{1\le i \le d}x_i$.
\end{definition}

For $p \ge 1$, this formula characterizes the standard $p$-norm restricted to the positive orthant $\mathbb{R}^d_+$. As a special case, the linear functional
\[
a_1 : \R_+^d \ni x \longmapsto \sum_{i=1}^d x_i \in \R_{+}
\]
is both a norm and antinorm, called the \textbf{1-antinorm}. See \cref{rmk:choice_1_antinorm} for additional considerations regarding its use in the numerical implementation.

%%%%%%%%%%%%%%%%%%%%%%%%%%%%%%%%%%%%%%%%%%%%%%%%%%%%%%%%%%%%%%%%%%%%%%%%%%%%%%%%%
\subsection{Matrix antinorms and Gelfand's limit} \label{sec:gelfand_limit} 

Let $K$ be a cone and $\cL(K)$ the set of all $d \times d$ real matrices for which $K$ is invariant. Then:
\[
\mathrm{int}(\cL(K)) = \left\{ A \in \cL(K) \: : \: K \text{ is strictly invariant for } A \right\}.
\]
Analogous to the well-known notion of operator norms for matrices, we introduce \textit{operator antinorms}:

\begin{definition}
	For any antinorm $a(\cdot)$ on $K$ and $A \in \cL(K)$, the \textbf{operator antinorm} of $A$ is defined as follows:
	\[
	a(A) := \inf \left\{ a(Ax) \: : \: x \in K \text{ with } a(x)=1 \right\} = \inf_{x \in \cA'} a(Ax).
	\]
\end{definition}

\begin{theorem}
	Let $a(\cdot)$ be an antinorm on $K$. Then the functional
	\[
	\cL(K) \ni A \longmapsto a(A) \in \R_{\ge 0}
	\]
	is an antinorm on $\cL(K)$, possibly discontinuous at the boundary $\partial \cL(K)$. Moreover, for all $A,B \in \cL(K)$ it holds:
	\begin{equation} \label{eq:supermul}
		a(AB) \ge a(A) a(B).
	\end{equation}
\end{theorem}

See \cite[Section 4]{GZ21} for the proof and additional properties of antinorms. From now on, assume $\cF = \{A_1,\ldots,A_m\}$ finite family of real $d \times d$ matrices sharing an invariant cone $K$.

\begin{definition}
	Let $a(\cdot)$ be any antinorm. Then, $\displaystyle a(\cF) := \min_{1 \le j \le m} a(A_j)$.
\end{definition}

The following result, proved in \cite[Proposition 6]{GP13}, guarantees that $a(\cF)$ is a lower bound for the lower spectral radius of $\cF$:

\begin{proposition}
	Let $\cF$, $K$ and $a(\cdot)$ be as above. Then it holds:
	\begin{equation}\label{eq.lowerbound1}
		a(\cF) \le \check{\rho}(\cF).
	\end{equation}
\end{proposition}

\begin{definition}
	An antinorm $a(\cdot)$ is \textbf{extremal} for $\cF$ if equality holds in \eqref{eq.lowerbound1}. 
\end{definition}

The existence of an extremal antinorm was established in \cite[Theorem 5]{GP13}. Although the proof is constructive, it is impractical for our purposes.

\begin{theorem} \label{thm:extremal_antinorm}
	Let $\cF$ and $K$ be as above. Then
	\begin{equation} \label{eq:optimization_problem}
		\check{\rho}(\cF) = \max\left\{ a(\cF) \: : \: a(\cdot) \text{ antinorm defined on } K\right\}.
	\end{equation}
\end{theorem}

Our goal is to find a lower bound for $\check{\rho}(\cF)$; so, solving the optimization problem \eqref{eq:optimization_problem} may not be the best strategy. Instead, fix an antinorm $a(\cdot)$ and define:
\[
\alpha_k (\cF) := \min_{P \in \Sigma_k} a(P)^{1/k}, \quad k \ge 1.
\]
For every $k \ge 1$, $\alpha_k(\cF)$ serves as a lower bound for $\check \rho(\cF)$:

\begin{theorem}
	Let $\cF$, $K$ and $a(\cdot)$ as above. Then, for all $k \ge 1$, it turns out
	\[
	\alpha_k(\cF) \le \check{\rho}(\cF).
	\]
	Moreover, if either $\check{\rho}(\cF)=0$ or $\alpha_1(\cF)>0$, then there exists
	\[
	\alpha(\cF) := \lim_{k \to + \infty} \alpha_k(\cF) = \sup_{k \ge 1} \alpha_k(\cF).
	\]
\end{theorem}

The proof is given in \cite[Theorem 5.2]{GZ21}. As a consequence, when the limit exists, we have the lower bound:
\begin{equation}\label{eq.cr1}
	\alpha(\cF) \le \check{\rho}(\cF).
\end{equation}
However, the value $\alpha(\cF)$ depends on the antinorm $a(\cdot)$ and, consequently, the inequality \eqref{eq.cr1} may be strict. In addition, the accuracy of $\alpha(\cF)$ in approximating the LSR for a given antinorm is unclear, suggesting that there is no clear method for selecting $a(\cdot)$ for any given family. Consequently, we restrict the class of admissible families. 

From now on, we assume $\cF$ \textbf{asymptotically rank-one}. This allows us to associate a unique decomposition of the space:
\[
\R^d = V_P \oplus H_P
\]
for every $P \in \Sigma(\cF)$. We define the \textit{leading set} of $\cF$ as
\[
\mathcal V (\cF) = \bigcup_{P \in \Sigma(\cF)} V_P,
\]
and the \textit{secondary set} as
\[
\cH(\cF) := \bigcup_{P \in \Sigma(\cF)} H_P,
\]
where $V_P$ is the leading eigenspace of $P$, and $H_P$ is defined as in \eqref{eq:secondary_eig}. If $K$ is an invariant cone for $\cF$, then by \cref{thm.cones} we have
\[
\mathrm{clos}\left(\mathcal V(\cF) \right) \subseteq K \cup -K \qquad \text{and} \qquad \mathrm{clos}\left(\mathcal H(\cF) \right) \cap \mathrm{int}(K) = \emptyset.
\]
We now state the main result of \cite{GZ21}, a Gelfand-type limit. Specifically, it establishes the equality in \eqref{eq.cr1}, providing a crucial tool for our algorithm.

\begin{theorem}\label{thm.gelfand} 
	Let $\cF$ be an asymptotically rank-one family of matrices and $K$ an invariant cone for $\cF$ such that
	\begin{equation}\label{eq.hpgelfand}
		K \cap \mathrm{clos}\left( \cH(\cF) \right) = \{0\}.
	\end{equation}
	Then, for any positive antinorm $a(\cdot)$, the following Gelfand's limit holds:
	\begin{equation}\label{eq:gelfand_limit}
		\alpha(\cF) := \lim_{k\to+\infty} \alpha_k(\cF) = \check{\rho}(\cF).
	\end{equation}
\end{theorem}

For $\cF = \{A\}$, we obtain a characterization of the spectral radius of any asymptotically rank-one matrix using any antinorm:

\begin{corollary}\label{thm.spectrallimit}
	Let $A$ be an asymptotically rank-one matrix and $K$ an invariant cone such that \eqref{eq.hpgelfand} holds. Then, there holds the spectral limit:
	\begin{equation}\label{eq.spradius}
		\rho(A) = \lim_{k \to+ \infty} \left[ a(A^k) \right]^{1/k}.
	\end{equation}
\end{corollary}

\begin{example}
	Consider the matrix family $\cF = \{A_1,A_2\}$ where
	\[
	A_1 = \begin{pmatrix}
		2 & 0 \\ 1 & 1
	\end{pmatrix} \qquad \text{and} \qquad A_2 = \begin{pmatrix}
		1 & 1 \\ 0 & 2
	\end{pmatrix}.
	\]
	Both $A_1$ and $A_2$ have spectrum $\Lambda = \{1,2\}$, secondary eigenvalue $\lambda = 1$, and corresponding eigenvectors $v(A_1) = (0,1)^T$ and $v(A_2) = (1,0)^T$. Thus,
	\[
	v(A_1),v(A_2) \in \cH(\cF) \implies \cH(\cF) \cap \R_+^d \neq \{0\},
	\]
	violating assumption \eqref{eq.hpgelfand}. The secondary eigenvectors of the transpose family $\cF^T = \{A_1^T,A_2^T\}$ (see \cref{rmk:self_dual_cone}), on the other hand, are 
	\[
	v(A_1)=v(A_2) = (-1,1)^T,
	\]
	so they do not belong to $K = \R_+^d$. As a result, we can apply \cref{thm.gelfand} to the transpose family, ensuring convergence of our algorithms.
\end{example}

We conclude with two important observations for nonnegative matrices. The first is that for families of strictly positive matrices, the Perron-Frobenius theorem \cite{HornJohnson1985MatrixAnalysis} ensures the asymptotic rank-one property. The second is that this is no longer true if $\cF$ includes matrices with zero entries. However, if we assume that $\cF$ is \textit{primitive}, namely $\exists k$ such that for all $P \in \Sigma_k(\cF)$, $P$ is positive, then the asymptotically rank-one property still holds, and as a result \cref{thm.gelfand} can be applied.

%%%%%%%%%%%%%%%%%%%%%%%%%%%%%%%%%%%%%%%%%%%%%%%%%%%%%%%%%%%%%%%%%%%%%%%%%%%%%%%%%%%%%%%%
\section{A Gripenberg-like algorithm for the LSR} \label{sec:algorithm}

In 1996, Gri\-pen\-berg \cite{Gripenberg1996} developed an algorithm to approximate the JSR of a matrix family $\cF$ to a specified accuracy $\delta>0$. The algorithm utilizes the bounds:
\[
\rho(P)^{1/k} \le \hat \rho(\cF) \le \|P\|^{1/k}, \qquad \text{for all } P \in \Sigma_k(\cF) \text{ and all } k \ge 1.
\]
This algorithm converges without assumptions on $\cF$, though the convergence rate depends on the chosen norm $\| \cdot \|$. The main novelty is its identification of subsets of products of degree $k$ (see \eqref{eq:Sigmak}), denoted by
\[
S_k \subset \Sigma_k(\cF) \text{ for every } k \in \N,
\]
which exclude unnecessary products for estimating the JSR and typically have much smaller cardinality than $\Sigma_k$.

In \cref{subsec:proof_main}, we present a Gripenberg-like algorithm that applies to the lower spectral radius and, specifically, prove convergence to a fixed accuracy $\delta>0$ independent of the initial antinorm choice.

In \cref{subsec:algorithm_enhanced}, we introduce an improved variant based on ideas from the polytopic algorithm \cite{GP13}. This enhancement adaptively refines the antinorm, boosting convergence rate and robustness, particularly when the extremal antinorm significantly differs from the initial one. We apply these ideas to develop an adaptive procedure improving the original Gripenberg algorithm for the JSR in \cref{subsec:adaptive_norm_algorithm}.

In \cref{subsec:perturbations}, we study the continuity properties of the LSR and develop a perturbation theory that allows to extend the algorithm's applicability to cases where the assumption \eqref{eq.hpgelfand} may not hold.

%%%%%%%%%%%%%%%%%%%%%%%%%%%%%%%%%%%%%%%%%%%%%%%%%%%%%%%%%%%%%%%%%%%%%%%%%%%%%%%%%%%%%%
\subsection{Proof of the main result} \label{subsec:proof_main}

This section presents and proves the article's main result (\cref{thm.main}), providing the theoretical foundation for our algorithms. 

Let $\cF = \{A_1,\ldots,A_m\}$ with an invariant cone $K$. Recall that $\Sigma_k := \Sigma_k(\cF)$ (see \eqref{eq:Sigmak}) denotes of all products of degree $k$. 

\subsubsection{Basic Methodology} \label{subsec:basic_methodology}

Given a target accuracy $\delta>0$ and a fixed antinorm $a(\cdot)$ on $K$, define:
\[
S_1 := \Sigma_1 = \cF.
\]
Additionally, set the initial lower and upper bounds for $\check \rho(\cF)$ as:
\[
\min_{A \in S_1} a(A) =: t_1 \le \check \rho(\cF) \le s_1 := \min_{A \in S_1} \rho(A).
\]
For $k > 1$, recursively define the subset of products of degree $k$:
\[
\scalebox{0.9}{$\displaystyle S_k := \left\{ P = \prod_{\ell=1}^k A_{i_\ell}, \ i_\ell \in \{1,\ldots,m\} 
	\: : \: \prod_{\ell=1}^{k-1} A_{i_\ell} \in S_{k-1} \text{ and } q(P) < s_{k-1} - \delta \right\}$},
\]
where
\begin{equation}\label{eq:definition_q}
	q(P) := \max_{1 \le j \le k} a\left( \prod_{\ell=1}^j A_{i_\ell} \right)^{1/j}.
\end{equation}
Next, update the sequences $t_k$ and $s_k$ as follows:
\[
\scalebox{0.9}{$\displaystyle t_k := \max\left\{ t_{k-1}, \min\left\{ s_k - \delta, \min_{P \in S_k} q(P) \right\} \right\}, \quad s_k := \min\left\{ s_{k-1}, \min_{P \in S_k }\rho(P)^{1/k} \right\}$}.
\]
Our main result establishes that $t_k$ and $s_k$ bound the lower spectral radius from below and above, respectively, with $s_k-t_k$ converging to the target accuracy $\delta$.

\begin{theorem} \label{thm.main}
	Let $\cF$ be an asymptotically rank-one family with an invariant cone $K$ satisfying \eqref{eq.hpgelfand}. Fix $\delta>0$ and choose any positive antinorm $a(\cdot)$ on $K$. Then
	\[
	t_k \le \check{\rho}(\cF) \le s_k
	\]
	for all $k \in \N$, and 
	\[
	\lim_{k\to+\infty} (s_k-t_k) = \delta.
	\]
\end{theorem}

\begin{proof}
	Combines \cref{lemma.a1}, \cref{lemma.a2}, and \cref{lemma.a3}.
\end{proof}

\begin{lemma}[Upper bound]\label{lemma.a1}
	For every $k \ge 1$, it holds $\check{\rho}(\cF) \le s_k$.
\end{lemma}

\begin{proof}
	For $k=1$, $s_1 \ge \check{\rho}(\cF)$ by definition. For $k>1$, we note that $S_k \subseteq \Sigma_k$. Since the minimization is monotone decreasing with respect to inclusion, we have:
	\[
	s_k \ge \min_{P \in S_k} \rho(P)^{1/k} \ge \min_{P \in \Sigma_k} \rho(P)^{1/k} \ge \check{\rho}(\cF).
	\]
\end{proof}

\begin{lemma}[Lower bound] \label{lemma.a2}
	For every $k \ge 1$, it holds $\check{\rho}(\cF) \ge t_k$.
\end{lemma}

\begin{proof}
	Consider an arbitrary product $R$ of degree $k$, say $R = \prod_{h= 1}^k A_{i_h} \in \Sigma_k$. Since the sequence $\{s_n\}_{n \in \mathbb N}$ is non-increasing, the definition of $t_k$ implies that
	\begin{equation} \label{eq:boundtk}
		\max_{1\le j \le k} a \left( \prod_{h = 1}^j A_{i_h} \right)^{1/j} =: q(R) \ge t_k.
	\end{equation}
	Consequently, there must be an index $j_\ast \in \{1,\ldots,k\}$ such that
	\begin{equation}\label{eq.aux1}
		a \left( \prod_{h = 1}^{j_\ast} A_{i_h} \right)^{1/j_\ast} \ge t_k.
	\end{equation}
	Consider now an arbitrary product $P$ of degree $N > k$ of the form
	\[
	P := \prod_{\ell = 1}^N A_{i_\ell}, \qquad i_\ell \in \{1,\ldots,m\} \text{ for all } \ell.
	\]
	Notice that we do not require $R$ to be a sub-product of $P$. Given the inequality \eqref{eq.aux1}, it is possible to construct a sequence of indices
	\[
	0 = j_0 < j_1 < j_2 < \cdots < j_p \le j_{p+1} = N
	\]
	such that, for every $0 \le r \le p-1$, there holds:
	\[
	j_{r+1}-j_r \le k \qquad \text{and} \qquad a \left( \prod_{\ell = j_r+1}^{j_{r+1}} A_{i_\ell} \right)^{\frac{1}{j_{r+1}-j_r}} \ge t_k.
	\]
	Using \eqref{eq:supermul}, the previous allows us to derive an estimate for the antinorm of $P$ as
	\[ \begin{aligned}
		a \left( P \right)^{1/N} & \ge \prod_{r=0}^{p-1} \left(a \left( \prod_{\ell = j_r+1}^{j_{r+1}} A_{i_\ell} \right)^{\frac{1}{j_{r+1}-j_r}}\right)^{\frac{j_{r+1}-j_r}{N}} \prod_{\ell = j_p+1}^N a(A_{i_\ell})^{1/N}
		\\[.5em] & \ge (t_k)^{j_p/N}  \inf_{A \in \cF} a(A)^{k/N},
	\end{aligned} \]
	where $j_p/N \to 1$ and $k/N \to 0$ as $N \to + \infty$. Since $P \in \Sigma_N$ is an arbitrary product, we can take the infimum over all such products of degree $N$ to conclude that
	\[
	\alpha_N(\cF) \ge (t_k)^{j_p/N} \inf_{A \in \cF} a(A)^{k/N}.
	\]
	By \cref{thm.gelfand}, the quantity $\alpha_N(\cF)$ converges to $\check \rho(\cF)$ as $N \to + \infty$. Therefore, passing to the limit the inequality above yields
	\[
	\check{\rho}(\cF) = \lim_{N \to + \infty} \alpha_N(\cF) \ge \lim_{N \to + \infty} \left[ (t_k)^{j_p/N} \inf_{A \in \cF} a(A)^{k/N} \right] = t_k,
	\]
	and this completes the proof.
\end{proof}

\begin{lemma}[Convergence] \label{lemma.a3}
	Under the assumptions of \cref{thm.main}, the difference between $s_k$ and $t_k$ converges to the fixed accuracy, i.e. $\lim_{k\to+\infty}(s_k-t_k)=\delta$.
\end{lemma}

\begin{proof}
	We argue by contradiction. Assume the limit is not $\delta$. Then there exists $\eps>0$ such that, defining $s_\infty := \lim_{n \to\infty} s_n$, we have
	\begin{equation}\label{eq.l31}
		t_k \le s_\infty - \delta - 2\eps \quad \text{for every } k \ge 1.
	\end{equation}
	Consequently, by the definition of $t_k$, the minimum $\min \left\{ s_k - \delta, \inf_{P\in S_k} q(P) \right\}$ is achieved by the second term, which can also be written as:
	\[
	\min_{P\in S_k} \max_{1 \le j \le k} a \left( \prod_{\ell=1}^j A_{i_\ell} \right)^{1/j}.
	\]
	In addition, for every $k \ge 1$, $\alpha_k(\cF) \le t_k$. Combining this with \eqref{eq.l31}, we find $\alpha_k(\cF) \le s_\infty - \delta - 2\eps$. Taking the limit as $k \to \infty$ and using \cref{thm.gelfand}, we deduce:
	\[
	\check{\rho}(\cF)  \le s_\infty - \delta - 2\eps.
	\]
	As $\eps>0$ is arbitrarily small, there exists $k \in \N$ and matrices $A_{i_1},\ldots,A_{i_k}$ in the family $\cF$ such that the following inequality holds:
	\begin{equation} \label{eq.main}
		\rho \left( \prod_{\ell = 1}^k A_{i_\ell}\right)^{1/k} < \check{\rho}(\cF) + \eps \le s_\infty - \delta - \eps.
	\end{equation}
	We claim that \eqref{eq.main} yields a contradiction. Consider $P = A_{i_1}\cdots A_{i_k}$ and define the infinite product $P^\infty := \prod_{j=1}^\infty P$, obtained by multiplying $P$ by itself infinitely many times. Our goal is to now identify a sequence of integers 
     \[
    0 = j_0 < j_1 < j_2 < \cdots
    \]
    such that, for $r \ge 0$, there holds:
	\[
	a \left( \prod_{\ell = j_r+1}^{j_{r+1}} A_{i_\ell} \right)^{1/(j_{r+1}-j_r)} \ge s_\infty - \delta - \eps.
	\]
	There are now two cases that we need to analyze separately:
	\begin{enumerate}[label=(\arabic*)]
		\item Infinitely many $j_r$ with this property exist. Applying \cref{thm.spectrallimit}, we calculate the spectral radius of $P$ as:
		\[
		\rho \left( P \right) = \lim_{N \to + \infty} a \left( \left( \prod_{\ell = 1}^k A_{i_\ell}\right)^N \right)^{1/N} = \lim_{N \to + \infty} a \left( P^N \right)^{1/N}.
		\]
		Thus, using the infinite sequence $j_0<j_1<\cdots$, we get 
  \[
  \rho \left( P \right)^{1/k} \ge s_\infty - \delta - \eps,
  \]
  contradicting \eqref{eq.main}.
		
		\item If not, then there exists an integer $j_p$, with $p \ge 0$, such that
		\[
		a \left( \prod_{\ell = j_p+1}^{j_p + N} A_{i_\ell} \right)^{1/N} < s_\infty - \delta - \eps \qquad \text{for } N \ge 1.
		\]
		The product $\prod_{\ell = j_p+1}^{j_p+k} A_{i_{\ell}}$ belongs to $S_k$ since the sequence of upper bounds $\{s_n\}_n$ is non-increasing and, consequently, it satisfies the inequality:
		\[
		a \left( \prod_{\ell = j_p+1}^{j_p + k} A_{i_\ell} \right)^{1/k} < s_\infty - \delta - \eps \le s_k - \delta - \eps.
		\]
		Furthermore, the sequence of matrices $(A_{i_{j_p+1}},\ldots,A_{i_{j_p+k}})$ is a cyclic permutation of $(A_{i_1},\ldots,A_{i_k})$. Therefore, since the spectral radius is invariant under cyclic permutations, applying \cref{thm.spectrallimit} yields:
		\[
		\rho \left( \prod_{\ell = j_p + 1}^{j_p+k} A_{i_\ell} \right)^{1/k} = \left[\lim_{N \to +\infty} a \left( P^N \right)^{1/N} \right]^{1/k} \ge s_\infty - \delta - \eps,
		\]
		again contradicting \eqref{eq.main}.
	\end{enumerate}
\end{proof}

%%%%%%%%%%%%%%%%%%%%%%%%%%%%%%%%%%%%%%%%%%%%%%%%%%%%%%%%%%%%%%%%%%%%%%%%%%%%%%%
\subsection{Adaptive methodology} \label{subsec:algorithm_enhanced}

We now introduce an adaptive procedure for refining the initial polytope antinorm. Unlike our previous approach, this strategy dynamically improves the antinorm throughout the algorithm. The outline of this adaptive procedure is as follows:
\begin{enumerate}[label=(\alph*)]
	\item Let $a^{(0)} := a$ be the initial polytope antinorm, with $V_0$ as the corresponding minimal vertex set (see \cref{def:pol_antinorm} for details).
	
	\item At step $k$, let $a^{(k-1)}$ and $V_{k-1}$ be the antinorm and minimal vertex set from step $k-1$. Set $V_k:=V_{k-1}$. For each matrix product $P \in S_k$, find $y \in \R^d$ that solves the problem:
	\[
	a^{(k-1)}(Py) = \min \left\{ a^{(k-1)}(Px) \: : \: x \in K \text{ s.t. } a^{(k-1)}(x)=1 \right\}.
	\]
	Add $Py$ to $V_k$ if it falls inside the polytope (i.e., $a^{(k-1)}(Py)\le1$); otherwise, discard it.
	
	\item After evaluating all matrix products in $S_k$, extract a minimal vertex set from $V_k$ and define $a^{(k)}(\cdot)$ as the corresponding polytope antinorm.
\end{enumerate}

\begin{proposition} \label{prop:improvement}
	Let $\cF$, $K$, $a(\cdot)$ and $s_k$ be as in \cref{thm.main}. Set $a^{(0)}:=a$, and define new lower bound and product set as:
	\[\begin{aligned}
		& t_k^\ast := \max\left\{ t_{k-1}^\ast, \min\left\{ s_k - \delta, \inf_{P \in S_k} q^{(k-1)}(P) \right\} \right\},
		\\
		& S_k^\ast := \left\{ P = \prod_{\ell=1}^k A_{i_\ell} \: : \: q^{(k-1)} (P) < s_{k-1} - \delta \text{ and } \prod_{\ell=1}^{k-1} A_{i_\ell} \in S_{k-1}^\ast \right\},
	\end{aligned}\]
	where $q^{(k-1)}(\cdot)$ is defined as in \eqref{eq:definition_q} with $a^{(k-1)}$ replacing $a(\cdot)$:
	\[
	q^{(k-1)}(P) := \max_{1 \le j \le k} a^{(k-1)}\left( \prod_{\ell=1}^j A_{i_\ell} \right)^{1/j}.
	\]
	Then $S_k^\ast \subseteq S_k$, $\check\rho(\cF) \ge t_k^\ast \ge t_k$, and $\lim_{k \to + \infty} (s_k - t_k^\ast) = \delta^\ast \le \delta$.
\end{proposition}

\begin{proof}
	This follows directly from the definition of $a^{(k)}$. Indeed, we obtain an increasing sequence of antinorms:
	\[
	a^{(0)}(P) \le \cdots \le a^{(k-1)}(P) \le a^{(k)}(P) \qquad \text{for every } P \in S_k.
	\]
\end{proof}

For numerical implementation, see \cref{subsec:adaptive_antinorm_algorithm}. We conclude this section with observations on the theoretical results:
\begin{itemize}
	\item While \cref{prop:improvement} only guarantees $\delta^\ast \le \delta$, empirical evidence suggests the adaptive procedure provides significantly more precise bounds.
	
	\item For simplicity, \cref{prop:improvement} defines $a^{(k)}$ only at the end of step $k$; however, the numerical algorithm refines the antinorm after incorporating each vertex, typically resulting in faster convergence - refer to \cref{subsec:adaptive_antinorm_algorithm}.
\end{itemize}

\subsection*{Eigenvector-based adaptive methodology} 

Inspired by the polytopic algorithm \cite{GP13}, we propose a potential enhancement to our adaptive procedure. If $\Sigma$ is a s.l.p. for $\cF$, the extremal antinorm can be obtained starting from the leading eigenvector of $\Sigma$. Taking this into account, we propose the following modification:
\begin{enumerate}[label=(\roman*)]
	\item At step $k$, identify all $P \in S_k^\ast$ that improve the upper bound.
	
	\item Choose a scaling factor $\vartheta > 1$. For each identified product $P$:
	\begin{itemize}
		\item let $v_P$ be its leading eigenvector;
		\item compute its antinorm $a(v_P)$;
		\item incorporate the rescaled vector $w_P := v_p / (\vartheta \cdot a(v_P))$ into the vertex set.
	\end{itemize}
\end{enumerate}

\begin{remark}
	While including eigenvectors in the vertex set may enhance convergence, this improvement is not guaranteed, unlike the adaptive procedure in \cref{prop:improvement}. Furthermore, the choice of the scaling parameter $\vartheta$ can significantly impact the algorithm's performance, as mentioned in \cref{subsec:algorithm_enhanced_eg}.
\end{remark}

%%%%%%%%%%%%%%%%%%%%%%%%%%%%%%%%%%%%%%%%%%%%%%%%%%%%%%%%%%%%%%%%%%%%%%%%%%%%%%%
\subsection{Continuity of the LSR and extendibility of our methodologies} \label{subsec:perturbations}

The continuity properties of the LSR are not fully understood. In our context, even the assumption of non-negativity for $\cF$ does not guarantee the continuity of $\check \rho$ in a neighborhood of $\cF$. Discontinuities may occur outside invariant cones (see, e.g., \cite{Protasov2006} and \cite{KashinKonyaginTemlyakov2023}). The following example, adapted from \cite{BochiMorris2015}, illustrates this issue:

\begin{example}
	Consider the family $\cF=\{A_1,A_2\}$ where
	\[
	A_1 = \begin{pmatrix}
		2 & 0 \\ 0 & 1/8
	\end{pmatrix} \quad \text{and} \quad A_2 = \begin{pmatrix}
		1 & 0 \\ 0 & 1
	\end{pmatrix},
	\]
	with $\check \rho(\cF)=1$. For $N \in \mathbb N$, consider the perturbed family $\cF_N:=\{A_1,R_{2\pi/N}\}$, where $R_{2\pi/N}$ is the rotation matrix of angle $2\pi/N$. For $m \ge 1$, we have
	\[
	(A_1)^m (R_{2\pi/N})^N = \begin{pmatrix} 2^m & 0 \\ 0 & 8^{-m} \end{pmatrix} \begin{pmatrix} 0 & - 1 \\ 1 & 0 \end{pmatrix} = \begin{pmatrix}
		0 & -2^m \\ 8^{-m} & 0
	\end{pmatrix},
	\]
	implying $\check\rho(\cF_N) = 1/2$ for all $N \in \mathbb N$, thus demonstrating the discontinuity of $\check \rho$ at $\cF$.
\end{example}

However, continuity can be ensured under a more restrictive condition: the existence of an \textit{invariant pair of embedded cones}, a concept introduced in \cite{ProtasovJungersBlondel2010}.

\begin{definition}
	A convex closed cone $K'$ is embedded in a cone $K$ if $K' \setminus \{0\} \subset \operatorname{int} K$. In this case, $\{K,K'\}$ is called an embedded pair.
\end{definition}

\begin{definition}
	An embedded pair $\{K,K'\}$ is an \textbf{invariant pair} for a matrix family $\cF$ if both $K$ and $K'$ are invariant for $\cF$.
\end{definition}

\begin{theorem} \label{thm.convergence_perturb}
	Let $\cF \subset \R^{d,d}$ be compact and $\cF_\eps$ be a sequence converging to $\cF$ in the Hausdorff metric. If $\cF$ admits an invariant pair of cones, then
	\begin{equation}\label{eq:lsr_continuity}
		\check \rho(\cF) = \lim_{\eps \to 0} \check \rho(\cF_\eps).
	\end{equation}
\end{theorem}

The proof of this result is given in \cite{Jungers2012AsymptoticProperties}. While identifying an invariant pair of cones is generally challenging, it is always possible and computationally inexpensive for strictly positive matrices:

\begin{lemma} \label{ch:posit}
	If $A > 0$ for all $A \in \cF$, then $\cF$ admits an invariant pair of cones.
\end{lemma}

\begin{proof}
	Let $A > 0$ and define $c(A) := \max_{1 \le j \le d} \left( \frac{\max_{1\le i \le d} a_{ij} }{\min_{1\le i \le d} a_{ij} } \right)$. By \cite[Corollary 2.14]{ProtasovJungersBlondel2010}, the cone defined as
	\[
	K_{c(A)} := \left\{ x \in \R_+^d \: : \: \max_{1\le i \le d}(x_i) \le c(A) \min_{1\le i \le d}(x_i) \right\}.
	\]
	is invariant for $\cF$ and embedded in $K:=\R_+^d$ by construction. Since $\cF$ is compact, $c:=\min_{A \in \cF} c(A)$ is well-defined. The corresponding cone $K' := K_c$ forms an invariant pair with $\R_+^d$, concluding the proof.
\end{proof}

This result has significant implications for the application and convergence of our algorithms. When $\cF$ includes matrices with zero entries, the assumptions of \cref{thm.gelfand} may not hold. However, we can address this issue by considering perturbed families of the form:
\[
\cF_\eps := \{A_1 + \eps \Delta_1,\ldots,A_m+\eps \Delta_m\}, \quad \text{where } \|\Delta_i\|_F=1 \text{ and } \Delta_i > 0.
\]
In these perturbed families, all matrices are positive, ensuring the convergence of our algorithms by the Perron-Frobenius theorem. Moreover, for every $\eps >0$, $\cF_\eps$ satisfies the assumption of \cref{ch:posit}, yielding:
\[
\check \rho(\cF)  = \lim_{\eps \to 0^+} \check \rho(\cF_\eps).
\]
This approach proves valuable in critical cases, such as the one discussed in \cref{subsec:critical_example}. However, there is a potential trade-off: for extremely small $\eps$, the algorithm may slow down significantly, leading to high computational costs. Empirical evidences suggests that the difference $\check \rho(\cF_\eps) - \check \rho(\cF)$ is of order $\eps$; hence, to observe this difference, we require $\delta \simeq \eps$, which is computationally expensive as $\eps \to 0$.

%%%%%%%%%%%%%%%%%%%%%%%%%%%%%%%%%%%%%%%%%%%%%%%%%%%%%%%%%%%%%%%%%%%%%%%%%%%%%%%%%%
\section{Algorithms} \label{sec:numerical_implementation}

This section focuses on the numerical implementation of the methodologies introduced in \cref{sec:algorithm}. We primarily address:
\begin{enumerate}[label=(\roman*)]
	\item \textit{Optimization of intermediate calculations} to minimize overall computational cost; for instance, when computing \eqref{eq:definition_q}.
	
	\item \textit{Development of efficient adaptive algorithms} that exploit the ideas from \cref{subsec:algorithm_enhanced} and \cref{subsec:algorithm_enhanced_eg} to refine the initial antinorm at each step.
\end{enumerate}

%%%%%%%%%%%%%%%%%%%%%%%%%%%%%%%%%%%%%%%%%%%%%%%%%%%%%%%%%%%%%%%%%%%%%%%%%%%%%%%%%%%
\subsection{Standard Algorithm (S)} \label{subsec:algorithm_s}

Let $\cF = \{A_1,\ldots,A_m\}$ be a family of non-negative real matrices for which $\mathbb R^d_+$ is an invariant cone. This section implements the basic methodology described in \cref{subsec:basic_methodology}, which utilizes a \textbf{fixed polytope antinorm} (\cref{def:pol_antinorm}).

We describe in details the computational cost of a polytope antinorm evaluation in \cref{subsec:adaptive_antinorm_algorithm}. It is important to note, however, that it requires solving a number of linear programming (LP) problems equal to the cardinality of the vertex set $V$. Conversely, for the $1$-antinorm, denoted $a_1(\cdot)$, we have the following formula:
\begin{equation}\label{eq.antinorm_1}
	\scalebox{1.0}{$a_1(A) = \min_{1 \le j \le d} \, \sum_{i=1}^d a_{ij}, \qquad \text{for every } A = (a_{ij})_{1\le i,j \le d} \in \R^{d,d}.$}
\end{equation}
Using \eqref{eq.antinorm_1} in the algorithm offers a significant computational advantage: calculating $a_1(A)$ becomes inexpensive as it does not require solving any LP problems. However, it is important to observe that the $1$-antinorm is often unsuitable; see \cref{subsec:il_ex} for an example of the $1$-antinorm converging too slowly, making it impractical.

\subsubsection{Pseudocode} \label{subsec:pseudocode_algorithm_s}

Algorithm \textbf{(S)} requires the following inputs: the elements of the family $\cF = \{A_1,\ldots,A_m\}$ and these parameters:

\begin{itemize}
	\item \textit{accuracy} ($\delta > 0$) as specified in \cref{thm.main}, and \textit{computation limit} ($M > 0$), the maximum allowed number of antinorm evaluations.
	
	\item \textit{vertex set} (V) which contains, as columns, the vertices corresponding to the initial polytope antinorm.
\end{itemize}

\begin{algorithm}[ht!]
	\small
	\caption{Algorithm (S) for Computing the Lower Spectral Radius}
	\label{alg.1}
	\begin{algorithmic}[1]
		\STATE{Set the initial lower and upper bounds for the LSR: $L=0$ and $H=+\infty$ \COMMENT{lower and upper bound initialized, respectively, to $0$ and $+\infty$}}
		\STATE{Let $S_1 := \cF$ \COMMENT{to store matrices of degree $1$, following the notation of \cref{thm.main}} and set $m = \# S_1$ \COMMENT{number of elements in the family $\cF$}}
		
		\FOR{$i=1$ to $m$}
		\STATE{compute the antinorm $a(A_i)$ and the spectral radius $\rho(A_i)$}
		\STATE{set $l(i) = a(A_i)$ \COMMENT{the vector $l$ stores all candidate lower bounds in the current step} and $H = \min\{H, \rho(A_i)\}$ \COMMENT{updates the upper bound if $\rho(A_i)$ improves it}}
		\ENDFOR
		\STATE{compute the lower bound after step one as $L = \min_{1 \le i \le m} l(i)$}
		\STATE{\textbf{Setting of iterations parameters:}}
		\STATE{set $\texttt{n} = 1$ \COMMENT{current degree}, $\texttt{n}_{op} = m$ \COMMENT{total number of antinorm evaluations}, $\texttt{J} = m$ \COMMENT{cardinality of $S_{\texttt{n}}$, which is $m$ after the first step}, and $\texttt{J}_{max} = \texttt{J}$ \COMMENT{to keep track of the maximum value of \texttt{J} in the main loop}}
		\STATE{set $\ell_{opt}=1$ \COMMENT{i.e., the degree for which the gap between lower and upper bounds is optimal} and $\ell_{slp}=1$ \COMMENT{i.e., degree yielding the optimal upper bound}}
		%\STATE{\textbf{Iterative loop of the algorithm:}}
		\WHILE{$H - L \ge \delta \quad \&\& \quad \texttt{n}_{op} \le M$}
		\STATE{set $H_{old} = H$, $L_{old} = L$ \COMMENT{store the current bounds} and $\texttt{J}_{new} = 0$}
		\STATE{set $\texttt{n}=\texttt{n}+1$ and initialize $S_{\texttt{n}} = [\cdot]$ \COMMENT{empty set to store matrix products of degree \texttt{n} for the next iteration, following the notation of \cref{thm.main}}}
		\FOR{$k = 1$ to $\texttt{J}$ \COMMENT{iterate on all products of degree $\texttt{n}-1$}}
		\FOR{$i = 1$ to $m$ \COMMENT{iterate on all elements of $\cF$}}
		\STATE{let $X_k$ be the $k$-th element in $S_{\texttt{n}-1}$ \COMMENT{note: $S_{\texttt{n}-1}$ has cardinality \texttt{J}}}
		\STATE{set $q = \texttt{J}_{new} + 1$ and consider the product $Y := X_k A_i$}
		\STATE{compute the antinorm $a(Y)$ and the spectral radius $\rho(Y)$}
		\STATE{set $l_{new}(q) = \max\{ l(k), (a(Y))^{1/\texttt{n}}\}$ and $H := \min\{ H, (\rho(Y))^{1/\texttt{n}} \}$}
		
		\IF{$l_{new}(q) < H - \delta$ \COMMENT{criterion for $Y$ to belongs to $S_{\texttt{n}}$}}
		\STATE{increase $\texttt{J}_{new} = \texttt{J}_{new} + 1$ and compute $L = \min\{L, l_{new}(q)\}$}
		\STATE{include $Y$ in the set $S_{\texttt{n}}$ \COMMENT{which is used in the next \textbf{while} iteration}}
		\ENDIF
		\ENDFOR
		\ENDFOR
		
		\STATE{compute $L = \max\{ L_{old}, \min\{ L, H - \delta \}\}$ and set $l := l_{new}$}
		\STATE{set $\texttt{n}_{op} = \texttt{n}_{op} + \texttt{J} \cdot m$, $\texttt{J} = \texttt{J}_{new}$ and $\texttt{J}_{max} = \max\{\texttt{J}, \texttt{J}_{max} \}$}
		
		\IF{$H - L < H_{old} - L_{old}$}
		\STATE{set $\ell_{opt} := \texttt{n}$ \COMMENT{store the degree yielding the best bounds gap}}
		\ENDIF
		\IF{$H<H_{old}$}
		\STATE{set $\ell_{slp} := \texttt{n}$ \COMMENT{store the degree yielding the best upper bound so far}}
		\ENDIF
		\ENDWHILE
		\STATE{\textbf{return} $\texttt{lsr}:=(L, H)$ and
			$\texttt{p} = (\ell_{opt}, \ell_{slp}, \texttt{n}, \texttt{n}_{op}, \texttt{J}_{max})$ \COMMENT{performance metrics}}
	\end{algorithmic}
\end{algorithm}

\begin{remark} \label{remark:alg_1}
	The algorithm outputs the best upper/lower bounds found, along with a \textbf{performance metric} \texttt{p} containing:
	
	\begin{itemize}
		\item $\ell_{opt}$: the optimal product degree (minimizing the gap $H-L$ between upper and lower bounds);
		\item $\ell_{slp}$: the product degree achieving the optimal (smallest) upper bound;
		\item $\texttt{n}$: the maximum product length considered by the algorithm;
		\item $\texttt{n}_{op}$: the number of antinorm evaluations performed;
		\item $\texttt{J}_{max}$: the maximum number of elements in any $S_k$, $k \in \{1,\ldots,\texttt{n}\}$. A small $\texttt{J}_{max}$ indicates the algorithm's ability to explore high-degree products while maintaining a reasonable number of antinorm evaluations.
	\end{itemize}
\end{remark}

To clarify the numerical implementation of \cref{alg.1}, we now highlight the following key lines of the pseudocode provided below:
\begin{enumerate}[label=(\alph*)]
	\item The initial step (lines 3--6) computes preliminary lower and upper bounds using all elements of $\cF$.
	
	\item The lower bound is updated directly (line 7) setting $L = \min_i l(i)$, allowing reuse of intermediate calculations stored in $l$ for efficient computation of \eqref{eq:definition_q} at line 19 as $\max\{l(k),a(Y)^{1/\texttt{n}}\}$.
	
	\item The iteration parameter $\texttt{J}$ tracks the cardinality of $S_{\texttt{n}}$ (\texttt{n} being the current degree), determining the length of the \textbf{for} loop in line 14.
	
	\item Restricting the number of antinorm evaluations (by requiring $\texttt{n}_{op} \le M$) ensures that the loop terminates even if the target accuracy $\delta$ cannot be achieved within reasonable time.
	
	\item In line 13, the set $S_{\texttt{n}}$ is initialized (as defined in \cref{thm.main}). This set stores products of length \texttt{n} for constructing products of length $\texttt{n}+1$ in the next step.
	
	\item The \texttt{if} statement (lines 20--23) determines whether $Y:=X_k A_i \in \Sigma_{\texttt{n}}$ also belongs to $S_{\texttt{n}}$. If so, it is stored for the next step in $S_{\texttt{n}}$; if not, it is discarded.
	
	\item The antinorm $a(Y)$ is computed using either \eqref{eq.antinorm_1} for the $1$-antinorm, or the method outlined in the subsequent \cref{subsec:adaptive_antinorm_algorithm} and implemented in the appendix (see \cref{supplement_sec:comp_antinorm}).
\end{enumerate}

\begin{remark} \label{remark:slp_vs_opt}
	The metric $\ell_{slp}$ tracks upper bound improvements, indicating candidate s.l.p. degree. In contrast, $\ell_{opt}$ updates when the gap between lower and upper bounds narrows, regardless of which bound changes. As a result, $\ell_{opt}$ is often larger than $\ell_{slp}$. For example, slow lower bound improvements can lead to $\ell_{opt} \gg \ell_{slp}$.
\end{remark}

\begin{remark} \label{rmk:choice_1_antinorm}
	The numerical implementation often uses $a_1$ as the intial (fixed) polytope antinorm, due to:
	\begin{itemize}
		\item Numerical stability of summing non-negative elements.
		
		\item Efficient computation of the antinorm of any matrix using \eqref{eq.antinorm_1}.
		
		\item Due to \cref{thm.gelfand}, the 1-antinorm is well-suited when $K=\R_+^d$ is a strictly invariant cone. This happens, for instance, when all matrices are positive.
	\end{itemize}
\end{remark}

Despite the guaranteed convergence of \cref{alg.1} to the target accuracy $\delta$, practical limitations exist. For example, using a fixed antinorm can significantly slow down lower bound convergence, resulting in excessively long computational times. Alternative strategies are the following:
\begin{enumerate}[label=(\roman*)]
	\item Initializing \cref{alg.1} with a different initial antinorm. For example, if $v$ is the leading eigenvector of some $A \in \cF$, the polytope antinorm (see \cref{def:pol_antinorm}) corresponding to the vertex set given by $v$ only is often effective.
	
	\item Developing an adaptive antinorm refinement algorithm (detailed in \cref{subsec:adaptive_antinorm_algorithm}), based on the theoretical results in \cref{prop:improvement}.
\end{enumerate}

%%%%%%%%%%%%%%%%%%%%%%%%%%%%%%%%%%%%%%%%%%%%%%%%%%%%%%%%%%%%%%%%%%%%%%%%%%%%
\subsection{Adaptive Algorithm (A)} \label{subsec:adaptive_antinorm_algorithm}

In this section, we develop the adaptive algorithm which is outlined in \cref{subsec:algorithm_enhanced} and draws inspiration from the polytopic algorithm. First, however, we briefly recall how to compute a polytope antinorm:

\subsubsection*{Polytope antinorm evaluation} 

Let $a(\cdot)$ be a polytope antinorm with a corresponding minimal vertex set $V = \{v_1,\ldots,v_p\}$ (see \cref{def:pol_antinorm}). Given a matrix $A$, to compute $a(A)$ we utilize the following formula:
\begin{equation}\label{eq.antinorm_generic}
	a(A) = \min_{v \in V} a(Av) = \min_{1\le j \le p} a(Av_j).
\end{equation}
This allows us to focus on determining $a(Av_j)$ for $j \in \{1,\dots,p\}$. Specifically, for $z := Av_j$, we solve the following LP problem:
\begin{equation}\label{eq:lp_problem}
	\min c_0 \quad \text{subject to } \begin{cases} c_0 z \ge \sum_{i=1}^p c_i v_i
		\\[.4em] \sum_{i=1}^p c_i \ge 1 \text{ with } c_i \ge 0 \text{ for every } i. \end{cases}
\end{equation}
The solution, denoted by $c_{V,z} := \min c_0$, is non-negative and potentially $+\infty$ if the system of inequalities has no solution. Consequently, we define $a(z)$ as:
\begin{equation} \label{eq:antinorm_vector}
	a(z) = \begin{cases}
		0 & \text{if } c_{V,z} = + \infty, \\
		+ \infty & \text{if } c_{V,z} = 0, \\
		1/c_{V,z} & \text{otherwise}.
	\end{cases}
\end{equation}
After solving all LP problems, we apply \eqref{eq.antinorm_generic} to determine $a(A)$. For an efficient numerical implementation, refer to \cref{supplement_sec:comp_antinorm}.

\begin{remark}
	Vertex pruning is crucial: minimizing the number of vertices reduces both the quantity and complexity of LP problems.
\end{remark}

To simplify the process of adding new vertices, we assume $\cF$ is normalized (i.e., its lower spectral radius is one). Our strategy is as follows:
\begin{enumerate}[label=(\arabic*)]
	\item Let $a^{(k-1)}$ be the polytope antinorm corresponding to $V_{k-1}$, as defined in \cref{prop:improvement}. Set $V_k := V_{k-1}$ and proceed to the $k$-th step. For each $P \in S_k^\ast$ evaluated by the algorithm, the candidate vertex $z := Av_i$ satisfies:
	\[
	a^{(k-1)}(P) = a^{(k-1)}(z).
	\]
	
	\item Include $z$ in $V_k$ based on the following criterion:
	\begin{itemize}
		\item If $a^{(k-1)}(z) > 1$, discard $z$ as it falls outside the polytope.
		\item If $a^{(k-1)}(z) \le 1$, incorporate $z$ into the vertex set.
	\end{itemize}
	
	\item After the $k$-th step, up to $\# S_k^\ast$ new vertices may be added, some potentially redundant. Apply the following pruning procedure:
	
	\begin{enumerate}[label=(\alph*)]
		\item Let $V_k = \{v_1,\ldots,v_q\}$, where $q > \# V_{k-1}$.
		\item For each $v_i \in V_k$, define $V_k^i$ as the matrix derived from $V_k$ by omitting $v_i$, and let $a_{v_i}(\cdot)$ be its corresponding polytope antinorm. If
		\begin{equation}\label{eq.antinorm_criterion}
			a_{v_i}(v_i) \ge 1 + \texttt{tol}
		\end{equation}
		remove $v_i$ from $V_k$; otherwise, retain $v_i$.
		\item Repeat this process for each $v_i \in V_k$ until either a minimal set is achieved or the rank of $V_k$ reduces to $1$.
	\end{enumerate}
	
	\item Define $a^{(k)}(\cdot)$ as the polytope antinorm with vertex set $V_k$. Proceed to the $(k+1)$-th step, repeating the procedure from $(1)$.
\end{enumerate}

In the numerical implementation (see \cref{alg.supplement.a} and \cref{supplement_sec:adaptive_alg}), we enhance this strategy. Each time we add a vertex, we immediately update the antinorm used in step (1) for the next vertex. This generates a non-decreasing sequence of antinorms:
\[
a_0^{(k-1)} \to a^{(k-1)}_1 \to \cdots \to a^{(k-1)}_\ell, \qquad \text{where } a^{(k-1)}_\ell = a^{(k)}.
\]
This approach offers an additional advantage: the algorithm might not add some redundant vertices before the pruning procedure even begins.

\begin{remark}\label{rmk:choice_antinorm}
	The initial antinorm choice significantly impacts the adaptive algorithm's performance. For instance, vertices defining the $1$-antinorm, due to their position on the boundary of the cone $\R_+^d$, often restrict the algorithm's ability to refine the polytope. This issue is clearly illustrated by \cref{fig:5.2} and the example provided in \cref{subsec:il_ex}.
\end{remark}

\begin{remark}\label{rmk:slp}
	Identifying s.l.p. candidates is straightforward to implement numerically, owing to their connection with optimal upper bounds - see \cref{supplementsec:identification_slp} for the pseudocode. Once the algorithm establishes $\ell_{slp}$ (see \cref{remark:alg_1}), it suffices to consider all products of degree $\ell_{slp}$ and compare their spectral radii.
\end{remark}

%%%%%%%%%%%%%%%%%%%%%%%%%%%%%%%%%%%%%%%%%%%%%%%%%%%%%%%%%%%%%%%%%%%%%%%%%%%%%%%%
\subsection{Adaptive Algorithm (E)} \label{subsec:algorithm_enhanced_eg}

We now present an extension of Algorithm \textbf{(A)} that potentially enhances the polytope antinorm by exploiting products that improve the upper bound. Within the main \textbf{while} loop of \cref{alg.1}, we update the current upper bound (line 19) using the following relation:
\begin{equation}\label{eq:adpative_1}
	H = \min\{ H , (\rho(Y))^{1/\texttt{n}}\}.
\end{equation}
This update allows us to identify all products $Y = X_k A_i$ that improve the current upper bound - specifically, those for which the minimum in \eqref{eq:adpative_1} is given by the spectral radius. We then use these products to refine the antinorm. More precisely, as discussed in \cref{subsec:algorithm_enhanced}, Algorithm \textbf{(E)} is derived from Algorithm \textbf{(A)} by implementing the following key modifications:
\begin{enumerate}[label=(\arabic*)]
	\item Set the scaling parameter $\vartheta$. While the optimal value is context-dependent, empirical evidence suggests that setting $\vartheta \in (1,1+\eps)$, where $\eps \ll 1$, is effective in most cases.
	
	\item Add vertices to the polytope antinorm as in Algorithm \textbf{(A)}. Additionally, when a product $P \in S_k^\ast$ improves the current upper bound according to \eqref{eq:adpative_1}, include in the vertex set the rescaled leading eigenvector $w$ of $P$:
	\[
	\widetilde w := (a(w) \cdot \vartheta)^{-1} w.
	\]
	
	\item Perform vertex pruning identical to Algorithm \textbf{(A)}. Proceed to the next step with the updated antinorm.
\end{enumerate}

The pseudocode of this adaptive algorithm is briefly discussed in the appendix (\cref{supplement_sec:adaptive_alg}). Results of numerical simulations are presented in \cref{sec:il_ex,sec:numerical_applications,sec:random}.

%%%%%%%%%%%%%%%%%%%%%%%%%%%%%%%%%%%%%%%%%%%%%%%%%%%%%%%%%%%%%%%%%%%%%%%%%%%%%%%%%%%
\section{Illustrative examples} \label{sec:il_ex}

In this section, to illustrate our algorithms, we consider two examples in detail. The first compares the performance with the polytopic algorithm \cite{GP13}, and the second shows how to deal with a critical case.

%%%%%%%%%%%%%%%%%%%%%%%%%%%%%%%%%%%%%%%%%%%%%%%%%%%%%%%%%%%%%%%%%%%%%%%%%%%%%%%%
\subsection{Comparison of the algorithms} \label{subsec:il_ex}
Let us consider the family of matrices
\[
\cF := \{A_1,A_2\} = \left\{ \begin{pmatrix}
	7 & 0 \\ 2 & 3
\end{pmatrix}, \begin{pmatrix}
	2 & 4 \\ 0 & 8
\end{pmatrix} \right\}.
\]
In \cite{BrunduZennaro2018} it is shown that the product $\Pi := A_1A_2(A_1^2A_2)^2$ is a s.l.p. and, therefore, that the lower spectral radius of $\cF$ is given by:
\[
\check{\rho}(\cF) = \rho(\Pi)^{1/8} = \left[ 4 \cdot \left( 213803 + \sqrt(44666192953) \right) \right]^{1/8} \approx 6.009313489.
\]
Applying the polytopic algorithm to $\widetilde \cF = \rho(\Pi)^{-1/8} \cdot \cF$, yields the extremal polytope antinorm illustrated in \cref{fig:5.2}. However, the rescaled family $\widetilde \cF$ violates the condition \eqref{eq.hpgelfand}, making our algorithms inapplicable; to address this issue, we simply work with the transpose family, which satisfies the assumption. 

We are now ready to examine the performances on this example of Algorithm \textbf{(S)} and its adaptive variants: Algorithm \textbf{(A)} and Algorithm \textbf{(E)}.

\subsection*{Algorithm (S)} 

\cref{tab:5.1} collects the results obtained with \cref{alg.1} when applied with the $1$-antinorm and $\delta=10^{-6}$ for various values of $M$. It contains the LSR bounds, along with some elements of the performance metric $\texttt{p}$ (see \cref{remark:alg_1}): optimal degree for convergence ($\ell_{opt}$), potential s.l.p. degree ($\ell_{slp}$), highest degree explored ($\texttt{n}$), and maximum number of products ($\texttt{J}_{max}$) for any given length.

\begin{table}[ht!]
	\centering
	\setlength\belowcaptionskip{-5pt}
	\begin{tabular}{
			@{}
			S[table-format=1.0e1]
			|
			S[table-format=1.6]
			S[table-format=1.6]
			|
			S[table-format=1.0]
			S[table-format=2.0]
			S[table-format=2.0]
			S[table-format=6.0]
			@{}
		}
		\toprule
		{$M$} & {\textbf{low.bound}} & {\textbf{up.bound}} & {$\ell_{slp}$} & {$\ell_{opt}$} & {\texttt{n}} & {$\texttt{J}_{max}$} \\
		\midrule
		5e1 & 0.985087 & 1.000025 & 5 &  7 &  7 &      9 \\
		1e3 & 0.995985 & 1.000000 & 8 & 17 & 18 &    139 \\
		5e3 & 0.997440 & 1.000000 & 8 & 25 & 25 &    647 \\
		1e4 & 0.997680 & 1.000000 & 8 & 27 & 28 &   1219 \\
		1e6 & 0.998927 & 1.000000 & 8 & 51 & 51 & 116265 \\
		\bottomrule
	\end{tabular}
	\caption{\cref{alg.1} applied with the $1$-antinorm and $\delta=10^{-6}$.}
	\label{tab:5.1}
\end{table}

The differences between the extremal polytope (see \cref{fig:5.2}) and the one generated by the $1$-antinorm suggest a slow convergence towards one. This is confirmed by the data: the lower bound improves from $M=50$ to $M=10^6$ (the latter required several hours) rather slowly, increasing by only approximately $1.4 \%$.

\subsection*{Algorithm (A)}

In this case, even taking $M=50$ leads to convergence:
\[
\texttt{lower bound} = \texttt{upper bound} = 1.
\]
Thus, the algorithm converges to an accuracy $\delta$ that is smaller than the $\eps$-precision of the machine. The performance metric $\texttt{p} = (8,8,8,54,5)$ shows that both $\ell_{opt}$ and $\ell_{slp}$ coincide with the correct s.l.p. degree ($8$). 

\subsection*{Algorithm (E)} 

The performance is similar to Algorithm {\bfseries (A)}, so we focus only on the behavior with respect to the scaling parameter $\vartheta$. Based on the data collected in \cref{tab:5.1.1}, we make the following observations:

\begin{table}[ht!]
	\centering
	\setlength\belowcaptionskip{-12pt}
	\begin{tabular}{
			@{}
			S[table-format=1.3]
			|
			S[table-format=1.0]
			S[table-format=1.0]
			c
			S[table-format=2.0]
			@{}
		}
		\toprule
		{$\vartheta$} & {\textbf{low.bound}} & {\textbf{up.bound}} & {$(\ell_{slp},\ell_{opt},\texttt{n},\texttt{n}_{op},\texttt{J}_{max})$} & {$\# V$} \\
		\midrule
		1.005 & 1 & 1 & (8,10,10,50,5)   &  7 \\
		1.055 & 1 & 1 & (5,7,7,30,93)  & 6 \\
		1.105 & 1 & 1 & (5,11,11,44,3)  & 8 \\
		1.155 & 1 & 1 & (8,10, 10,64,5) & 7 \\
		1.205 & 1 & 1 & (8,18,18,172,9) & 20 \\
		1.605 & 1 & 1 & (8,20,20,362,24) & 53 \\
		\bottomrule
	\end{tabular}
	\caption{Algorithm {\bfseries (E)} applied with different values of the scaling parameter $\vartheta$.}
	\label{tab:5.1.1}
\end{table}

\begin{itemize}
	\item A small scaling parameter, such as $\vartheta=1.005$, is beneficial. Convergence is achieved with a small number of antinorm evaluations ($\texttt{n}_{op}=50$) and the s.l.p. degree is correctly identified $\ell_{slp}=8$. It returns a polytope with $7$ vertices, close to the extremal one - see \cref{fig:5.2} (\textbf{left}).
	
	\item Larger values of $\vartheta$ either fail to identify the s.l.p. degree or require more antinorm evaluations, and thus computational cost, for convergence. Furthermore, as shown in \cref{fig:5.2} (\textbf{right}), the vertex count increases significantly and the polytope starts to resemble the whole cone $\R_+^d$, making it inefficient.
\end{itemize}

\begin{figure}[ht!]
	\centering
	\setlength\belowcaptionskip{-8pt}
	\begin{subfigure}{0.49\textwidth}
		\centering
		\includegraphics[scale=0.6]{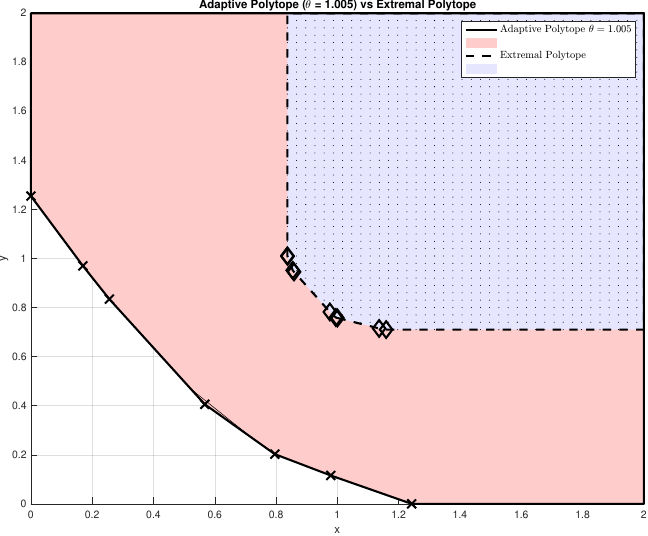}
	\end{subfigure}
	\hfill 
	\begin{subfigure}{0.49\textwidth}
		\centering
		\includegraphics[scale=0.65]{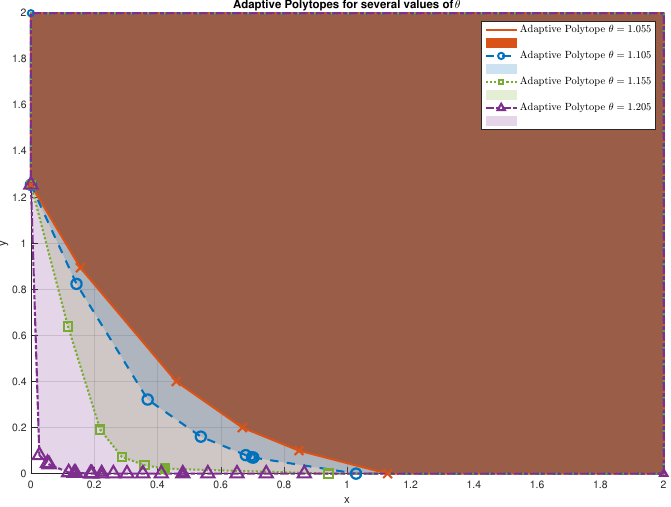}
	\end{subfigure}
	\caption{Comparison between extremal and adaptive ($\vartheta=1.005$) polytopes (\textbf{left}). Comparison of adaptive polytopes generated for different values of $\vartheta$ (\textbf{right}). }
	\label{fig:5.2}
\end{figure}

%%%%%%%%%%%%%%%%%%%%%%%%%%%%%%%%%%%%%%%%%%%%%%%%%%%%%%%%%%%%%%%%%%%%%%%%%%%%%%%%%%
\subsection{Critical example} \label{subsec:critical_example}

A critical case for our algorithms is a family $\cF$ violating the assumptions of \cref{thm.main}. Specifically, there are secondary eigenvectors lying on the boundary of $\R_+^d$ and some $A_i \in \cF$ have non-simple leading eigenvalues. For instance, consider the family $\cF=\{A_1,A_2\}$, where
\[\scalebox{0.85}{$
	A_1 = \begin{pmatrix}
		5 & 1 & 0 & 0 \\
		0 & 5 & 2 & 0 \\
		0 & 0 & 3 & 1 \\
		0 & 0 & 0 & 2 \\
	\end{pmatrix} \qquad \text{and} \qquad A_2 = \begin{pmatrix}
		1 & 2 & 3 & 4 \\
		0 & 2 & 5 & 6 \\
		0 & 0 & 3 & 7 \\
		0 & 0 & 0 & 4 \\
	\end{pmatrix}.$}
\]
Since $A_1$ has a non-simple leading eigenvalue ($\lambda = 5$) and $\cF$ does not satisfy \eqref{eq.hpgelfand}, \cref{thm.gelfand} does not apply. Thus, the identity \eqref{eq:gelfand_limit} fails for an arbitrary antinorm. To illustrate this issue, let $a_1(\cdot)$ be the $1$-antinorm. By \eqref{eq.antinorm_1}, there holds
\[
a_1(A_1) = 3 \qquad \text{and} \qquad a_1(A_2) = 1,
\]
which means that the lower bound is $1$ after the first step. It is easy to verify that $a_1(A_2^k) = 1$ for all $k \in \N$; therefore, the lower bound remains stuck at one, showing that the $1$-antinorm cannot satisfy the Gelfand's limit \eqref{eq:gelfand_limit}.

\begin{remark} \label{rmk:critical_example_problem}
	Even the adaptive algorithms \textbf{(A)} and \textbf{(E)} fail to converge when starting with the $1$-antinorm. This is because vectors with small $2$-norms (order of $10^{-12}$) are incorporated into the vertex set, leading to $S_6^\ast = \emptyset$. Consequently, no matrix products are selected for evaluation in the subsequent step, preventing the algorithm from progressing further and thus failing to converge to $\delta$.    
\end{remark}

Since $\check \rho(\cF)=3$ and $\Pi = A_1^3 A_2^4$ is a s.l.p., we consider the rescaled family $\widetilde \cF := \rho(\cF)^{-1} \cdot \cF$ and the vertex set $V = \{ v_1 \}$, where $v_1$ is the leading eigenvector of $\widetilde \Pi := \rho(\Pi)^{-1} \cdot \Pi$. In this case, it turns out that:

\begin{itemize}
	\item Algorithm \textbf{(S)} is no longer stuck, but converges too slowly: it returns a lower bound of $0.889231$ with $M=10^6$, making it impractical.
	
	\item Algorithm \textbf{(A)} converges to the target accuracy $\delta=10^{-6}$ within $\texttt{n}_{op}=512$ antinorm evaluations. However, the vertex set grows to $1020$ vertices despite the pruning procedure, making the {\ttfamily linprog} function (solving LP problems) increasingly slower.
	
	\item Algorithm \textbf{(E)} fails to return a lower bound because eigenvectors of the form $(\star, 0, 0, 0)^T$ are included into the vertex set, causing the algorithm to stop prematurely for the same reason described in \cref{rmk:critical_example_problem}.  
\end{itemize}

To avoid relying on a priori knowledge of an s.l.p., we can use a regularization technique (see \cref{subsec:perturbations}) by introducing small perturbations of the form:
\begin{equation}\label{eq:pert_cr}
\widetilde \cF_\eps^\Delta = \widetilde \cF + \eps \{ \Delta_1,\Delta_2 \},
\end{equation}
where $\eps > 0$ and $\Delta_1,\Delta_2$ are non-negative matrices with $\| \Delta_i \|_F = 1$. By \cref{thm.convergence_perturb}, the lower spectral radius is right-continuous; therefore, we have:
\[
\lim_{\eps \to 0^+} \check \rho(\cF_\eps^\Delta) = \check \rho(\cF).
\]
The main advantage of this approach is that we can choose the perturbations to ensure that $\widetilde A_1+\eps \Delta_1$ and $\widetilde A_2 + \eps \Delta_2$ are strictly positive for all $\eps > 0$. By the Perron-Frobenius theorem, this choice guarantees the convergence of our algorithms.

\begin{table}[ht!]
	\centering
	\setlength\belowcaptionskip{-5pt}
	\begin{tabular}{
			@{}
			S[table-format=1.1e-1]
			|
			S[table-format=1.6]
			S[table-format=1.6]
			|
			S[table-format=1.6]
			S[table-format=1.6]
			@{}
		}
		\toprule
		{$\eps$} & \multicolumn{2}{c}{Algorithm \textbf{(A)}} & \multicolumn{2}{c}{Algorithm \textbf{(E)}} \\
		\cmidrule(lr){2-3} \cmidrule(lr){4-5}
		& {\textbf{low.bound}} & {\textbf{up.bound}} & {\textbf{low.bound}} & {\textbf{up.bound}} \\
		\midrule
		1e-3 & 1.001520 & 1.109109 & 1.028864 & 1.109109 \\
		1e-4 & 0.999286 & 1.046703 & 0.999286 & 1.046703 \\
		1e-5 & 0.997308 & 1.018235 & 0.997517 & 1.018235 \\
		1e-6 & 0.994291 & 1.005660 & 0.994587 & 1.005660 \\
		1e-7 & 0.998546 & 1.001086 & 0.999128 & 1.001086 \\
		\bottomrule
	\end{tabular}
	\caption{Algorithms \textbf{(A)} and \textbf{(E)} initialized with the $1$-antinorm. Here Algorithm \textbf{(E)} uses the scaling parameter $\vartheta = 1.005$ for simplicity.}
	\label{tab:5.2.1}
\end{table}

Let us now focus on numerical simulation results obtained by applying Algorithm \textbf{(E)} to perturbations of $\cF$. In particular, we consider several values of the scaling parameter $\vartheta$ and perturbations of the form \eqref{eq:pert_cr}:

\begin{table}[ht]
	\centering
	\setlength\belowcaptionskip{-8pt}
	\scalebox{0.9}{
	\begin{tabular}{
			@{}
			S[table-format=1.3]
			|
			S[table-format=1.0e-1]
			|
			S[table-format=1.6]
			S[table-format=1.6]
			S[table-format=2.0]
			S[table-format=2.0]
			S[table-format=2.0]
			@{}
		}
		\toprule
		$\vartheta$ & $\eps$ & \textbf{low.bound} & \textbf{up.bound} & $\#V$ & $\ell_{slp}$ & $\ell_{opt}$ \\ \midrule
		1.005       & 1e-3   & 1.062580           & 1.101026          & 15    & 12           & 16           \\
		& 1e-5   & 1.000000           & 1.016912          & 14    & 7            & 12           \\
		& 1e-7   & 0.999271           & 1.000948          & 27    & 7            & 13           \\ \midrule
		1.015       & 1e-3   & 1.0647885          & 1.104739          & 10    & 7            & 16           \\
		& 1e-5   & 1.000000           & 1.017197          & 15    & 7            & 12           \\
		& 1e-7   & 0.999271           & 1.000968          & 51    & 7            & 13           \\ \midrule
		1.025       & 1e-3   & 1.065196           & 1.103720          & 11    & 5            & 16           \\
		& 1e-5   & 1.000000           & 1.017378          & 14    & 7            & 12           \\
		& 1e-7   & 0.999271           & 1.000994          & 74    & 7            & 13           \\ \midrule
		1.035       & 1e-3   & 1.069077           & 1.110830          & 15    & 7            & 16           \\
		& 1e-5   & 1.000000           & 1.019048          & 15    & 7            & 12           \\
		& 1e-7   & 0.999271           & 1.001193          & 86    & 7            & 13           \\ \midrule
		1.045       & 1e-3   & 1.068030           & 1.110830          & 15    & 7            & 15           \\
		& 1e-5   & 1.000000           & 1.019048          & 15    & 7            & 12           \\
		& 1e-7   & 0.998857           & 1.001193          & 98    & 7            & 13           \\ \bottomrule
	\end{tabular}}
	\caption{Algorithm \textbf{(E)} applied for various values of the scaling parameter $\vartheta$.}
	\label{tab:s1}
\end{table}

We conclude with few interesting observations drawn from this data, specifically in comparison with the results presented in \cref{tab:5.2.1}:
\begin{itemize}
	\item The s.l.p. degree of $7$ is correctly identified for all $\vartheta$ values when $\eps \le 10^{-5}$.
	
	\item For $\eps \in \{10^{-3},10^{-5}\}$, the impact of $\vartheta$ is minimal. This is further confirmed by the similarity in the number of vertices for different $\vartheta$ values.
	
	\item The scaling parameter $\vartheta = 1.005$ is optimal in two aspects: 
	\begin{itemize}
		\item It yields the smallest gap between upper and lower bounds for $\eps = 10^{-7}$.
		\item The resulting adaptive antinorm has the smallest number of vertices.
	\end{itemize}
	
	\item For $\eps = 10^{-7}$, larger values of $\vartheta$ become problematic. The vertex count grows substantially, while the gap between upper and lower bounds widens.
\end{itemize}

%%%%%%%%%%%%%%%%%%%%%%%%%%%%%%%%%%%%%%%%%%%%%%%%%%%%%%%%%%%%%%%%%%%%%%%%%%%%%%%%%%%
\section{Numerical applications} \label{sec:numerical_applications}

In this section, we discuss two applications in combinatorial analysis that have already been addressed with the polytopic algorithm in \cite[Section 8]{GP13}. Our goal here is different: we aim to find accurate bounds for the lower spectral radius as efficientlt as possible. Thus, it is not essential to identify extremal antinorms; for instance, as shown in \cref{tab:5.3}, adaptive polytopes can have significantly fewer vertices than extremal ones while yielding accurate approximations.

%%%%%%%%%%%%%%%%%%%%%%%%%%%%%%%%%%%%%%%%%%%%%%%%%%%%%%%%%%%%%%%%%%%%%%%%%%%%%%%%%%%
\subsection{The density of ones in the Pascal rhombus}

The Pascal rhombus is an extension of Pascal's triangle. Unlike the latter, which is generated by summing adjacent numbers in the above row forming a triangular array, the Pascal rhombus creates a two-dimensional rhomboidal shape array (\cite{Goldwasser1999}). The elements of the Pascal rhombus can also be characterized by a linear recurrence relation on polynomials:
\[
\begin{cases}
	p_n(x) = (x^2+x+1) \cdot p_{n-1}(x) + x^2 \cdot p_{n-2}(x), & \text{for } n \ge 2,
	\\ p_0(x) = 1, \quad p_1(x) = x^2+x+1. \end{cases}
\]
The asymptotic growth of $w_n$, which denotes the number of odd coefficients in the polynomial $p_n$, as shown in \cite{finch2008odd}, depends on the JSR $\hat \rho(\cF)$ and the LSR $\check \rho(\cF)$ of the matrix family $\cF := \{A_1,A_2\}$, where:
\[
A_1 = \begin{pmatrix} 0 & 1 & 0 & 0 & 0 \\ 1 & 0 & 2 & 0 & 0 \\ 0 & 0 & 0 & 0 & 0 \\ 0 & 1 & 0 & 0 & 1 \\ 0 & 0 & 0 & 2 & 1\end{pmatrix} \qquad \text{and} \qquad A_2 = \begin{pmatrix} 1 & 0 & 2 & 0 & 0 \\ 0 & 0 & 0 & 2 & 1 \\ 1 & 1 & 0 & 0 & 0 \\ 0 & 0 & 0 & 0 & 0 \\ 0 & 1 & 0 & 0 & 0 \end{pmatrix}.
\]
While the JSR ($\hat \rho(\cF) = 2$) can be easily determined, computing the lower spectral radius is quite challenging. In \cite{GP13}, the polytopic algorithm is employed to determine that $\Pi = A_1^3 A_2^3$ is a s.l.p., resulting in $\check \rho(\cF) = \rho(\Pi)^{1/6} \approx 1.6376$. We now compare this to the outcomes obtained by applying our algorithms:

\begin{itemize}
	\item \textbf{Algorithm (S).} Setting $M=10^4$ and $\delta=10^{-6}$, and applying \cref{alg.1} with the $1$-antinorm to the rescaled family $\widetilde \cF := \check \rho^{-1}(\cF) \cdot \{ A_1, A_2 \}$, yields:
	\[
	0.9987 \le \check \rho(\widetilde \cF) \le 1.
	\]
	The upper bound equals one because $M$ is large enough for the algorithm to evaluate products of degree six, $\widetilde \Pi := \rho(\Pi)^{-1} \cdot \Pi$ included. In contrast with the example in \cref{subsec:il_ex}, increasing $M$ improves the lower bound:
	\[
	M = 10^6 \implies (1-\texttt{lower bound}) \approx 10^{-4}.
	\]
	However, the difference between the extremal polytope (which includes $8$ vertices) and the polytope corresponding to the $1$-antinorm slows down convergence as $M$ increases.
	
	\item \textbf{Algorithm (A).} Applying the adaptive algorithm using the $1$-antinorm as the initial one, yields a significantly more precise and efficient outcome:
	\[
	M = 500 \implies (1 - \texttt{lower bound}) \approx 10^{-6}=\delta.
	\]
	The algorithm identifies the s.l.p. degree $\ell_{slp}=6$ correctly. If we apply the algorithm again using the refined antinorm (which has $29$ vertices) as the initial one, we achieve a much higher precision ($\delta=10^{-11}$) within the same amount of antinorm evaluations $\texttt{n}_{op} \approx 50$. In this case, the algorithm correctly returns all cyclic permutations of $\widetilde \Pi$ as s.l.p. candidates.
	
	\item \textbf{Algorithm (E)}. Outperforms Algorithm \textbf{(A)} in terms of speed and accuracy; e.g., it achieves an accuracy of $\delta= 10^{-12}$ within $150$ antinorm evaluations.
\end{itemize}

%%%%%%%%%%%%%%%%%%%%%%%%%%%%%%%%%%%%%%%%%%%%%%%%%%%%%%%%%%%%%%%%%%%%%%%%%%%%%%%%%%%

\subsection{Euler binary partition functions}\label{subsec:euler}

Let $r \geq 2$ integer. The Euler partition function $b(k) := b(r,k)$ is defined as the number of distinct binary representations:
\[
k = \sum_{j=0}^{+\infty} d_j 2^j, \quad \text{where } d_j \in \{0,\ldots,r-1\} \text{ for all } j \in \N.
\]
The function's asymptotic behavior - see \cite{Protasov2000} for an overview - is characterized in \cite{Reznick1990}, for $r$ odd, via the JSR ($\hat \rho$) and LSR ($\check\rho$) of a matrix family $\cF$ as follows:
\[
\limsup_{k \to + \infty} \frac{\log b(k)}{\log k} = \log_2 \hat \rho(\cF)
\qquad \text{and} \qquad \liminf_{k \to + \infty} \frac{\log b(k)}{\log k} = \log_2 \check \rho(\cF).
\]
More precisely, the family $\cF$ is given by two $(r-1)\times (r-1)$ matrices which, for $s \in \{1,2\}$, are defined by the following relation:
\[
(A_s)_{ij} = \begin{cases} 1 & \text{if } i+2-s \le 2j \le i + r - s + 1, \\ 0 & \text{otherwise}. \end{cases}
\]
In \cite{GP13}, the polytopic algorithm demonstrates that for $r$ odd (up to $41$), one of the matrices $A_1$ or $A_1A_2$ acts as a s.m.p. while the other as a s.l.p. 

\begin{remark}
	Algorithm \textbf{(A)} generally performs better than Algorithm \textbf{(E)}, as shown by the data in \cref{tab:5.3} and \cref{tab:supplement_euler_antinorm1_alge}. This is likely due to Algorithm \textbf{(E)} including several eigenvectors of matrix products of degree $\ge 3$ in the vertex set, which reduces the impact of the leading eigenvector of the s.l.p. ($A_1$ or $A_1A_2$).
\end{remark}

In this context, the $1$-antinorm is not recommended as the cone $\R_+^d$ is not strictly invariant for $\cF$. Instead, we use the polytope antinorm corresponding to $V=\{v_1\}$, where $v_1$ is the leading eigenvector of $A_1$. Moreover, since taking $M$ large has limited benefit when the s.l.p. degree is low, a more effective approach is the following:

\begin{enumerate}[label=(\arabic*)]
	\item Apply Algorithm \textbf{(A)}/\textbf{(E)} with $M$ small (e.g., $M=10$) to establish a preliminary bound $L \le \check \rho(\cF)\le H$. In this case, the upper bound $H$ already coincides with the LSR because the s.l.p. degree is either one or two.
	
	\item To prevent large entries from appearing while exploring long products, use the preliminary lower bound to rescale the family of matrices. In particular, consider the matrix family $\widetilde \cF := L^{-1} \cdot \cF$.
	
	\item Set $\delta>0$, $j=1$ (iteration count), and $M$. Apply Algorithm \textbf{(A)}/\textbf{(E)} to obtain a new lower bound ($L_{new}$), upper bound ($H_{new}$, which does not improve further in this case) and vertex set ($V_{new}$). 
	
	\item Check the stopping criterion:
	\begin{equation}\label{eq:alg_stop}
		H_{new} - L_{new} < \delta \qquad \text{or} \qquad \left| L_{new} - L \right| < \delta.
	\end{equation}
	If \eqref{eq:alg_stop} is satisfied or $j = \texttt{maxIter}$ (which is fixed a priori), the procedure terminates, returning $(L_{new}, H_{new})$. Otherwise, increase $j$ by one, rescale the family with $L_{new}$, set $L=L_{new}$ and $H=H_{new}$, and re-apply Algorithm \textbf{(A)}/\textbf{(E)} starting with the refined antinorm corresponding to $V_{new}$.
\end{enumerate}

In \cref{tab:5.3}, we summarize the results obtained by applying this strategy initializing Algorithm \textbf{(A)} with the parameters $M = 100$ and $\texttt{maxIter} = 20$. It provides a comparison between the vertex count of \textit{adaptive polytopes} ($\# V_{adap}$) and \textit{extremal polytopes} ($\# V_{ext}$), the latter taken from \cite[Section 8]{GP13}. The table also includes a time column (\textbf{T}[s]), rounded to the nearest second, showing the algorithm's efficiency even for higher dimensions.

\begin{table}[ht!]
	\centering
	\setlength\belowcaptionskip{-8pt}
	\scalebox{0.8}{
		\begin{tabular}{
				@{}
				S[table-format=2.0]
				|
				S[table-format=2.6]
				S[table-format=2.6]
				S[table-format=3.0]
				S[table-format=2.0]
				S[table-format=3.0]
				|
				S[table-format=2.0]
				|
				S[table-format=2.6]
				S[table-format=2.6]
				S[table-format=3.0]
				S[table-format=2.0]
				S[table-format=3.0]
				@{}
			}
			\toprule
			{\textbf{r}} & {\textbf{low.bound}} & {\textbf{up.bound}} & {\textbf{T}{[}s{]}} & {$\#V_{adap}$} & {$\# V_{ext}$} &
			{\textbf{r}} & {\textbf{low.bound}} & {\textbf{up.bound}} & {\textbf{T}{[}s{]}} & {$\#V_{adap}$} & {$\# V_{ext}$} \\
			\midrule
			7 & 3.490363 & 3.491891 &  78 & 12 &  14 & 25 & 12.499932 & 12.499971 & 393 & 42 &  58 \\
			9 & 4.493916 & 4.494493 & 200 & 11 &  17 & 27 & 13.499891 & 13.499938 & 354 & 24 &  37 \\
			11 & 5.496584 & 5.497043 & 349 & 19 &  24 & 29 & 14.499962 & 14.499982 & 344 & 24 &  43 \\
			13 & 6.498560 & 6.498947 & 335 & 21 &  28 & 31 & 15.499997 & 15.499999 & 302 & 23 &  34 \\
			15 & 7.499776 & 7.499842 & 261 & 23 &  23 & 33 & 16.499998 & 16.499999 & 289 & 25 &  55 \\
			17 & 8.499875 & 8.499904 & 260 & 28 &  30 & 35 & 17.499980 & 17.499997 & 310 & 22 & 102 \\
			19 & 9.499646 & 9.499789 & 340 & 24 &  46 & 37 & 18.499963 & 18.499994 & 341 & 25 & 113 \\
			21 & 10.499695 & 10.499813 & 372 & 27 &  50 & 39 & 19.499986 & 19.499994 & 292 & 18 &  59 \\
			23 & 11.499848 & 11.499894 & 317 & 33 &  31 & 41 & 20.499983 & 20.499997 & 336 & 23 & 120 \\
			\bottomrule
	\end{tabular}}
	\caption{Data obtained applying the procedure above with Algorithm \textbf{(A)} and initial antinorm determined by the leading eigenvector of $A_1$.}
	\label{tab:5.3}
\end{table}

 In conclusion, we make the following observations:
\begin{itemize}
	\item Reducing the maximum iteration number leads to slightly less precise approximations, but the computational time decreases significantly - see  \cref{tab:supplement_euler_lessprec}.
	
	\item Since Algorithm \textbf{(A)} restarts with a rescaled matrix family after each iteration, the target accuracy $\delta$ is \textit{relative} rather than absolute. 
	
	\item The vertex set corresponding to the extremal antinorm is typically larger than the adaptive one given by Algorithm \textbf{(A)}. This difference becomes more significant in higher dimensions - refer to \cref{tab:supplement_euler_larger}.
\end{itemize}

It is worth noting that Algorithm \textbf{(E)} can be applied to very large values of $r$, returning accurate bounds in a short time-frame. For instance:

\begin{table}[ht!]
	\centering
	\setlength\belowcaptionskip{-8pt}
	\begin{tabular}{
			@{}
			S[table-format=4.0]
			|
			S[table-format=3.10]
			S[table-format=3.10]
			c
			S[table-format=3.0]
			S[table-format=3.0]
			@{}
		}
		\toprule
		{\textbf{r}} & {\textbf{low.bound}} & {\textbf{up.bound}} & {\textbf{s.l.p.}} & {$\# V_{adapt}$} & {$\# V_{extr}$} \\
		\midrule
		251 & 125.4999994835 & 125.5000000000 & $A_1A_2$ & 27 & 24 \\
		501 & 250.4999998897 & 250.5000000000 & $A_1$ & 97 & 11 \\
		751 & 375.4999999117 & 375.5000000000 & $A_1 A_2$ & 125 & 5 \\
		1001 & 500.4998890105 & 500.5000000000 & $A_1$ & 4 & 11 \\
		1501 & 750.4999999250 & 750.5000000000 & $A_1$ & 30 & 7 \\
		\bottomrule
	\end{tabular}
	\caption{Algorithm \textbf{(E)} applied to large values of $r$. Comparison between the adaptive polytopes and the extremal polytopes (given by the polytope algorithm).}
	\label{tab:supplement_euler_larger}
\end{table}

Our algorithm is able to identify the correct s.l.p. ($A_1$ or $A_1 A_2$), while keeping low the number of vertices of the adaptive antinorm. In addition, convergence to $\delta=10^{-6}$ is typically achieved in a few minutes.

Let us now compare Algorithm \textbf{(E)} applied with two different choice of the initial antinorm: the $1$-antinorm and the antinorm derived from the leading eigenvector of $A_1$.

\begin{table}[ht!]
	\centering
	\setlength\belowcaptionskip{-10pt}
	\scalebox{0.9}{
		\begin{tabular}{
				@{}
				S[table-format=2.0]
				|
				S[table-format=2.6]
				S[table-format=2.6]
				S[table-format=2.0]
				|
				S[table-format=2.6]
				S[table-format=2.6]
				S[table-format=3.0]
				@{}
			}
			\toprule
			& \multicolumn{3}{c|}{\textbf{$1$-antinorm}} & \multicolumn{3}{c}{\textbf{antinorm $V=\{v_1\}$}} \\
			\midrule
			{\textbf{r}} & {\textbf{low.bound}} & {\textbf{up.bound}} & {$\# V$}  & {\textbf{low.bound}} & {\textbf{up.bound}} & {$\# V$} \\
			\midrule
			7   & 3.491891   & 3.491891  & 62        & 3.491891 & 3.491891  & 38       \\
			9   & 4.487566   & 4.494493  & 63       & 4.494370   & 4.494493  & 16       \\
			11  & 5.489837   & 5.497043  & 66     & 5.497043  & 5.497043  & 23      \\
			13  & 6.479278   & 6.498947  & 67        & 6.495736   & 6.498947  & 36  \\
			15  & 7.457363   & 7.499842  & 69     & 7.499830   & 7.499842  & 28          \\
			17  & 8.461750   & 8.499904  & 72      & 8.499888 & 8.499904  & 29       \\
			19  & 9.458195   & 9.499789  & 74    & 9.499789 & 9.499789  & 126   \\
			21  & 10.467492  & 10.499813 & 76        & 10.498826  & 10.499813 & 16       \\
			23  & 11.460798  & 11.499893 & 72        & 11.499856  & 11.499894 & 29     \\
			25  & 12.450722  & 12.499971 & 80          & 12.499824  & 12.499971 & 17      \\
			27  & 13.445389  & 13.499938 & 81        & 13.499723  & 13.499938 & 17        \\
			29  & 14.414215  & 14.499982 & 84         & 14.499089  & 14.499982 & 16       \\
			31  & 15.380684  & 15.500000 & 80          & 15.499982  & 15.500000 & 17        \\
			33  & 16.402404  & 16.500000 & 88         & 16.499194  & 16.500000 & 5       \\
			35  & 17.424716  & 17.499997 & 90      & 17.499432 & 17.499997 & 9       \\
			37  & 18.431588  & 18.499994 & 92        & 18.499185  & 18.499994 & 19        \\
			39  & 19.423864  & 19.499994 & 91      & 19.423864  & 19.499994 & 8          \\
			41  & 20.414314  & 20.499997 & 95       &   20.499542 & 20.499997 & 9 \\
			\bottomrule
	\end{tabular}}
	\caption{Algorithm \textbf{(E)} with two different choices of the initial antinorm.}
	\label{tab:supplement_euler_antinorm1_alge}
\end{table}

\begin{enumerate}[label=(\roman*)]
	\item Algorithm \textbf{(E)} starting with the $1$-antinorm performs poorly. The accuracy decreases significantly as $r$ increases, and in some cases, it takes hours for a standard laptop to achieve such poor results.
	
	\item The other choice is generally more accurate and requires less computational time, though it does not always converge to $\delta=10^{-6}$.
\end{enumerate}

Finally, we apply Algorithm \textbf{(A)} using the same initial antinorm as \cref{tab:5.3}, but limiting the computational cost (reducing the number of allowed iterations). The results, presented in \cref{tab:supplement_euler_lessprec} below, while less precise than \cref{tab:5.3} due to less allowed iterations, remains relatively accurate compared to \cref{tab:supplement_euler_antinorm1_alge}:

\begin{table}[ht!]
	\centering
	\setlength\belowcaptionskip{-8pt}
	\scalebox{0.83}{
		\begin{tabular}{@{}S[table-format=2.0]|S[table-format=2.6]S[table-format=2.6]S[table-format=2.0]S[table-format=3.0]|S[table-format=2.0]|S[table-format=2.6]S[table-format=2.6]S[table-format=2.0]S[table-format=3.0]@{}}
			\toprule
			{\textbf{r}} & {\textbf{low.bound}} & {\textbf{up.bound}} & {$\# V_{adp}$} & {$\# V_{ext}$} & {\textbf{r}} & {\textbf{low.bound}} & {\textbf{up.bound}} & {$\# V_{adp}$} & {$\# V_{ext}$} \\
			\midrule
			7 & 3.489753 & 3.491891 & 11 & 14 & 25 & 12.499378 & 12.499971 & 9 & 58 \\
			9 & 4.492477 & 4.494493 & 12 & 17 & 27 & 13.498726 & 13.499938 & 11 & 37 \\
			11 & 5.492146 & 5.497043 & 15 & 24 & 29 & 14.499118 & 14.499982 & 11 & 43 \\
			13 & 6.495201 & 6.498947 & 16 & 28 & 31 & 15.499456 & 15.499999 & 10 & 34 \\
			15 & 7.498230 & 7.499842 & 11 & 23 & 33 & 16.499702 & 16.499999 & 7 & 55 \\
			17 & 8.499265 & 8.499904 & 11 & 30 & 35 & 17.499298 & 17.499996 & 11 & 102 \\
			19 & 9.497710 & 9.499789 & 8 & 46 & 37 & 18.499373 & 18.499994 & 8 & 113 \\
			21 & 10.497841 & 10.499813 & 8 & 50 & 39 & 19.499542 & 19.499994 & 7 & 59 \\
			23 & 11.498875 & 11.499894 & 11 & 31 & 41 & 20.499736 & 20.499997 & 8 & 120 \\
			\bottomrule
	\end{tabular}}
	\caption{Algorithm \textbf{(A)} with $\texttt{maxIter} = 8$ and $M=50$.}
	\label{tab:supplement_euler_lessprec}
\end{table}

%%%%%%%%%%%%%%%%%%%%%%%%%%%%%%%%%%%%%%%%%%%%%%%%%%
\section{Simulations on randomly generated families} \label{sec:random}

In this section, we examine the performance of Algorithm \textbf{(E)} applied to randomly generated families of matrices, including both full matrices and sparse matrices with several sparsity densities $\rho$. The section is organized as follows:
\begin{enumerate}[label=(\alph*)]
	\item In \cref{random:alg_s}, we briefly discuss the unsuitability of the non-adaptive algorithm (\cref{alg.1}) for such classes of matrices.
	\item In \cref{subsec:alg_e_random}, we apply Algorithm \textbf{(E)} to randomly generated families with scaling parameter fixed ($\vartheta=1.005$) and examine the results obtained.
	\item In \cref{subsec:supplemennt_reg}, we apply Algorithm \textbf{(E)} to sparse families with low sparsity densities ($\rho \le 0.25$) using the regularization technique involving perturbations.
	\item In \cref{subsec:opt_theta}, we discuss the choice of the scaling parameter $\vartheta$ for Algorithm \textbf{(E)} by performing a grid-search optimization.
	\item In \cref{subsections:random}, we consider examples of families including more than two matrices and compare the results obtained.
\end{enumerate}

\subsection{Algorithm (S) analysis} \label{random:alg_s}

Algorithm \textbf{(S)} exhibits limitations when applied to sparse matrices with sparsity density $\rho \le 0.5$, mainly due to the many zero entries. While it can be utilized for (almost) full matrices, its performance in high-dimensional spaces is poor compared to adaptive algorithms, as shown in the table below:

\begin{table}[ht]
	\centering
	\setlength\belowcaptionskip{-8pt}
	\begin{tabular}{
			@{}
			S[table-format=3.0]
			|
			S[table-format=2.6]
			S[table-format=3.6]
			|
			S[table-format=2.6]
			S[table-format=2.6]
			@{}
		}
		\toprule
		& \multicolumn{2}{c|}{\textbf{$1$-antinorm}} & \multicolumn{2}{c}{\textbf{antinorm $V=\{v_1\}$}} \\
		\midrule
		{\textbf{d}} & {\textbf{low.bound}} & {\textbf{up.bound}} & {\textbf{low.bound}} & {\textbf{up.bound}} \\
		\midrule
		25   & 12.014676  & 12.204626  & 12.124379 & 12.396828 \\
		50   & 24.213542   & 24.545062  & 24.267634 & 24.573435 \\
		100   & 49.671633 & 50.133269 & 49.632976 & 49.632976 \\
		150   & 74.001873 & 74.530903 & 74.672414 & 75.389293 \\
		200   & 96.186436 & 100.017694 & 98.924848 & 99.742000  \\
		\bottomrule
	\end{tabular}
	\caption{Algorithm \textbf{(S)} with two different choices of the initial antinorm.}
	\label{tab:supplement_algorithms_1}
\end{table}

All simulations were conducted with a maximum of $M=10^5$ antinorm evaluations. This limitation was necessary as high-dimensional cases ($d\ge 50$) required several hours for algorithm termination, especially with the $1$-antinorm.

As expected, using the antinorm given by the leading eigenvector of a matrix family element yielded more accurate results than the $1$-antinorm. However, in no instance did we achieve convergence within a reasonable time-frame, making Algorithm \textbf{(S)} unsuitable for randomly generated families. 

Conversely, as shown in subsequent sections, adaptive algorithms significantly outperforms Algorithm \textbf{(S)} in both speed and accuracy.

\subsection{Algorithm (E) analysis} \label{subsec:alg_e_random}

We now examine the performance of Algorithm \textbf{(E)} applied with $\vartheta=1.005$ to randomly generated families of matrices, including both full matrices and sparse matrices with several sparsity densities $\rho$.

\begin{remark} 
	For randomly generated families, the s.l.p. degree is typically small; thus, we employ the same strategy outlined in \cref{subsec:euler}.
\end{remark}

\begin{table}[ht!]
	\centering
	\setlength\belowcaptionskip{-8pt}
	\begin{tabular}{
			@{}
			S[table-format=3.0]
			|
			S[table-format=3.6]
			S[table-format=3.6]
			c
			S[table-format=2.0]
			S[table-format=3.0]
			@{}
		}
		\toprule
		{\textbf{d}} & {\textbf{low.bound}} & {\textbf{up.bound}} & {$(\ell_{slp},\ell_{opt},\texttt{n},\texttt{n}_{op},\texttt{J}_{max})$} & {$\# V$} & {\textbf{T{[}s{]}}} \\
		\midrule
		25 & 12.555472 & 12.555472 & $(1,3,3,14,8)$   &  4 &  46 \\
		25 & 12.228911 & 12.228911 & $(2,7,7,254,128)$ & 30 & 108 \\
		50 & 24.647721 & 24.647721 & $(1,3,3,14,8)$    & 12 &  32 \\
		50 & 24.792006 & 24.792006 & $(3,5,5,62,32)$   &  4 &  38 \\
		100 & 49.419854 & 49.419909 & $(3,5,7,254,128)$ &  4 &  78 \\
		100 & 49.660798 & 49.660798 & $(1,2,2,6,4)$     & 32 & 118 \\
		150 & 74.813865 & 74.813865 & $(1,6,6,126,64)$  &  4 &  39 \\
		150 & 74.883378 & 74.883379 & $(1,3,3,14,8)$    &  4 &  82 \\
		200 & 100.172004 & 100.172005 & $(3,11,11,154,20)$ &  5 &  32 \\
		200 & 99.915573 & 99.915574 & $(1,5,5,62,32)$   &  4 &  51 \\
		\bottomrule
	\end{tabular}
	\caption{Algorithm {\bfseries (E)} applied to randomly generated families of strictly positive matrices. The initial antinorm is given by the leading eigenvector of $A_1$.}
	\label{tab:5.4}
\end{table}

As shown in \cref{tab:5.4}, Algorithm {\bfseries (E)} establishes bounds for the LSR of $d \times d$ random families with strictly nonzero entries, achieving high accuracy. 

The efficiency of the algorithm is demonstrated by low values of $\texttt{J}_{max}$, which represents the maximum number of factors evaluated at any degree. When the set of examined products $S_k^\ast$ remains small, the algorithm can analyze longer matrix products within a small number of antinorm evaluations (e.g., $M=200$). Furthermore, the adaptive procedure typically adds only a few vertices to the initial antinorm, keeping the computational cost reasonable while improving the convergence rate.

Interestingly, the algorithm correctly identified a s.l.p. in most cases (except for the first $200\times 200$ family), which we verified using the polytopic algorithm.

\begin{table}[!ht]
	\setlength\belowcaptionskip{-5pt}
	\centering
	\begin{tabular}{
			@{}
			S[table-format=2.2]
			|
			S[table-format=3.0]
			|
			S[table-format=3.6]
			S[table-format=3.6]
			c
			S[table-format=2.0]
			@{}
		}
		\toprule
		{\textbf{$\rho$}} & \textbf{d} & {\textbf{low.bound}} & {\textbf{up.bound}} & {$(\ell_{slp},\ell_{opt},\texttt{n},\texttt{n}_{op},\texttt{J}_{max})$} & {$\#V$} \\
		\midrule
		0.25 & 25 & 2.531959 & 2.531959 & $(3,7,7,254,128)$ & 8 \\
		& 50 & 5.425970 & 5.425970 & $(2,8,8,510,256)$ & 29 \\
		& 100 & 11.22993 & 11.22993 & $(2,6,6,126,64)$ & 19 \\
		& 150 & 16.338661 & 16.338661 & $(2,8,8,510,256)$ & 8 \\
		& 200 & 22.194604 & 22.194604 & $(2,4,4,30,16)$ & 8 \\
		\midrule
		0.50 & 25 & 4.700123 & 4.700123 & $(4,6,6,126,64)$ & 13 \\
		& 50 &9.771280 & 9.771280 & $(3,4,4,30,16)$ & 4 \\
		& 100 &19.748468 & 19.748468 & $(1,4,4,30,16)$ & 7 \\
		& 150 &29.331904 & 29.331904 & $(2,7,7,254,128)$ & 7 \\
		& 200 &39.459047 & 39.459047 & $(2,6,6,126,64)$ & 12 \\
		\midrule
		0.75& 25 & 6.480977 & 6.480977 & $(2,6,6,126,64)$ & 4 \\
		& 50 &12.811969 & 12.811969 & $(1,5,5,62,32)$ & 4 \\
		& 100 &26.476497 & 26.476497 & $(2,5,5,62,32)$ & 4 \\
		& 150 &39.447342 & 39.447342 & $(1,6,6,126,64)$ & 4 \\
		& 200 &52.924237 & 52.924237 & $(1,4,4,30,16)$ & 4 \\
		\bottomrule
	\end{tabular}
	\caption{Algorithm {\bfseries (E)} applied to randomly generated families of matrices with sparsity density $\rho$. The initial antinorm is given by the leading eigenvector of $A_1$.}
	\label{tab:5.5}
\end{table}

\cref{tab:5.5} gathers the results obtained when applying the same procedure on sparse matrices with sparsity density $\rho$. Below, we remark some interesting observations:
\begin{itemize}
	\item For both full and sparse families, we use the initial antinorm given by the leading eigenvector of $A_1$. This choice is especially important for sparse matrices, as the $1$-antinorm is not suitable with numerous zero entries.
	
	\item When $\rho = 0.25$, the algorithm typically incorporates more vertices compared to higher density or full matrices. Furthermore, a relative accuracy of $\delta = 10^{-6}$ generally requires a higher number of operations ($\texttt{n}_{op}$). This is likely due to $S_k^\ast$, which represents the set of matrix products of degree $k$ that are evaluated, having cardinality close to $\Sigma_k$. Nonetheless, the algorithm converges within a reasonable time frame and successfully identifies s.l.p. in all tested dimensions $d$.
	
	\item When $\rho \in \{0.5,0.75\}$, the algorithm exhibits significantly faster convergence and incorporates fewer vertices. %Additionally, convergence is generally achieved with fewer antinorm evaluations, improvement in efficiency, and s.l.p. are always identified correctly.
\end{itemize}

\subsection{Regularization}\label{subsec:supplemennt_reg}

In this section, we apply Algorithm \textbf{(E)} to families regularized through perturbations, utilizing the methodology outlined in \cref{subsec:euler}. We focus our analysis on sparse families with low sparsity densities ($\rho \le 0.25$) and compare the results to \cref{subsec:alg_e_random}, where no regularization technique was used. 

\begin{table}[ht!]
	\centering
	\setlength\belowcaptionskip{-8pt}
	\begin{tabular}{
			@{}
			S[table-format=3.0]
			|
			S[table-format=1.0e-1]
			|
			S[table-format=1.6]
			S[table-format=1.6]
			S[table-format=3.0]
			S[table-format=2.0]
			S[table-format=2.0]
			@{}
		}
		\toprule
		{\textbf{d}} & {$\eps$} & {\textbf{low.bound}} & {\textbf{up.bound}}  & {$\#V$} & {$\ell_{slp}$} & {$\ell_{opt}$} \\
		\midrule
		25 & 1e-3 & 0.773553 & 0.773554 & 38 & 4 & 14 \\
		& 1e-5 & 0.772502 & 0.772503 & 46 & 1 & 16 \\
		& 1e-7 & 0.772492 & 0.772493 & 43 & 1 & 42 \\
		\midrule
		50 & 1e-3 & 2.223289 & 2.223293  & 52 & 1 & 11 \\
		& 1e-5 & 2.222668 & 2.222702  & 59 & 1 & 9 \\
		& 1e-7 & 2.222460 & 2.222696 & 52 & 1 & 9 \\
		\midrule
		100 & 1e-3 & 4.608453 & 4.608458 & 55 & 3 & 10 \\
		& 1e-5 & 4.606516 & 4.607851  & 47 & 3 & 9 \\
		& 1e-7 & 4.607840 & 4.607845 & 57 & 3 & 10 \\
		\midrule
		200 & 1e-3 & 9.476520 & 9.476530 & 8 & 2 & 21 \\
		& 1e-5 & 9.475914 & 9.475924 & 8 & 3 & 9 \\
		& 1e-7 & 9.475908 & 9.475918 & 10 & 1 & 13 \\
		\bottomrule
	\end{tabular}
	\caption{Sparse families with sparsity density $\rho = 0.10$. Algorithm \textbf{(E)} is applied starting from the leading eigenvector of $A_1$.}
	\label{tab:reg_supp_2}
\end{table}

\begin{remark}Let us make a few observations about the data gathered above:
	\begin{itemize}
		\item Optimizing the scaling parameter $\vartheta$ is not a concern here. Thus, we use the fixed value $\vartheta = 1.010$ for all simulations for consistency.
		
		\item Due to the low sparsity density, the $1$-antinorm is not suitable to initialize the algorithm as it may get stuck. Instead, we use the polytope antinorm derived from the leading eigenvector of $A_1$.
	\end{itemize}
\end{remark}

We now make a few observations about the data gathered in \cref{tab:reg_supp_2}. In this case, we know that the LSR is right-continuous, so we expect that
\[
\check\rho(\cF) \approx \check\rho(\cF_{\eps})
\]
with $\eps=10^{-7}$ is the correct approximation with a target accuracy of $\delta=10^{-6}$. In other words, a natural question arising here is the following: is it true that
\begin{equation}\label{eq:formula_ee}
	\check \rho(\cF_{\eps}) = \check \rho (\cF) + C \eps + \mathcal O(\eps^2)
\end{equation}
for some positive constant $C$? Another interesting question is the following: if the matrix product $\Pi_\eps = \prod_{j=1}^k (A_{i_j}+\eps \Delta_{i_j})$ is a s.l.p. for $\cF_\eps$, is it true that
\[
\Pi := \Pi_\eps \, \big|_{\eps = 0} = \prod_{j=1}^k \left(A_{i_j} + \eps \Delta_{i_j}\right) \, \Big|_{\eps=0} = \prod_{j=1}^k A_{i_j}
\]
is a s.l.p. for the unperturbed family $\cF$? For example, when $d=25$ and $\eps=10^{-7}$, the algorithm finds the s.l.p. 
\[
\Pi_{10^{-7}} = A_1 + 10^{-7} \Delta_1.
\]
Next, we apply the polytopic algorithm to $\Pi=A_1$ and $\cF$ and find that $\Pi$ is indeed a s.l.p. for the unperturbed family $\cF$ and the extremal antinorm consists of $71$ vertices. As a result, the LSR of $\cF$ can be computed exactly as $\check \rho(\cF) = \rho(A_1)$. Thus,
\[
\check \rho(\cF_{\eps}) - \check \rho(\cF) = \rho(A_1+\eps \Delta_1)-\rho(A_1) \approx \eps
\]
holds when $\eps=10^{-7}$, which means that \eqref{eq:formula_ee} is verified in this case, making this approximation correct for a target accuracy of $\delta=10^{-6}$.

In conclusion, the regularization technique performs well for any value of the dimension $d$ tested. The target accuracy of $\delta = 10^{-6}$ is achieved within minutes, while the number of vertices remains relatively small.

\begin{table}[ht!]
	\centering
	\setlength\belowcaptionskip{-8pt}
	\begin{tabular}{
			@{}
			S[table-format=3.0]
			|
			S[table-format=1.0e-1]
			|
			S[table-format=2.6]
			S[table-format=2.6]
			S[table-format=3.0]
			S[table-format=3.0]
			S[table-format=2.0]
			S[table-format=2.0]
			@{}
		}
		\toprule
		{\textbf{d}} & {$\eps$} & {\textbf{low.bound}} & {\textbf{up.bound}} & {\textbf{T[s]}} & {$\#V$} & {$\ell_{slp}$} & {$\ell_{opt}$} \\
		\midrule
		25 & 1e-3 & 2.293235 & 2.293237 & 120 & 22 & 6 & 10 \\
		& 1e-5 & 2.292609 & 2.292612 & 93 & 24 & 4 & 7 \\
		& 1e-7 & 2.292603 & 2.292605 & 80 & 26 & 2 & 8 \\
		\midrule
		50 & 1e-3 & 4.484785 & 4.484790  & 56 & 15 & 1 & 8 \\
		& 1e-5 & 4.484173 & 4.484177 & 70 & 18 & 1 & 7 \\
		& 1e-7 & 4.484167 & 4.484171 & 74 & 20 & 2 & 7 \\
		\midrule
		100 & 1e-3 & 4.608453 & 4.608458 & 40 & 55 & 3 & 10 \\
		& 1e-5 & 4.606516 & 4.607851 & 45 & 47 & 3 & 9 \\
		& 1e-7 & 4.607840 & 4.607845 & 32 & 57 & 3 & 10 \\
		\midrule
		200 & 1e-3 & 17.940780 & 17.940797  & 120 & 9 & 3 & 8 \\
		& 1e-5 & 17.940174 & 17.940191 & 131 & 10 & 2 & 7 \\
		& 1e-7 & 17.940167 & 17.940185 & 176 & 10 & 2 & 10 \\
		\bottomrule
	\end{tabular}
	\caption{Regularization of sparse families with sparsity density $\rho = 0.20$.}
	\label{tab:reg_supp_1}
\end{table}

The results for a sparsity density of $\rho = 0.2$ show a marked improvement in both accuracy and computational efficiency compared to those presented in \cref{tab:reg_supp_2}.

Additionally, if we examine closely the case of $d=50$, we see that the algorithm returns the following s.l.p. candidates depending on the value of $\eps$:
\[
\Pi_{\eps} = \begin{cases} A_2 + \eps \Delta_2 & \text{for } \eps \in \{10^{-3},10^{-5}\},\\ (A_2+\eps \Delta_2)^2 & \text{for } \eps = 10^{-7}. \end{cases}
\]
Theoretically, given that the spectral radius property $\rho(A^2)=\rho(A)^2$ holds for any matrix $A$, the algorithm should have returned $A+\eps \Delta_2$ for all $\eps$ tested. The observed discrepancy for $\eps=10^{-7}$ is likely due to numerical inaccuracies in computing the square root, which updates the upper bound even though it does not change.

That said, we applied the polytopic algorithm to $A_2$ directly. This confirmed that $A_2$ is indeed a s.l.p. for the unperturbed family. The extremal antinorm for this case has only $9$ vertices, which is fewer than all the adaptive ones.

\begin{table}[ht!]
	\centering
	\setlength\belowcaptionskip{-8pt}
	\begin{tabular}{
			@{}
			S[table-format=3.0]
			|
			S[table-format=1.0e-1]
			|
			S[table-format=2.6]
			S[table-format=2.6]
			S[table-format=4.0]
			S[table-format=3.0]
			S[table-format=1.0]
			S[table-format=2.0]
			@{}
		}
		\toprule
		{\textbf{d}} & {$\eps$} & {\textbf{low.bound}} & {\textbf{up.bound}}  & {\textbf{T}[s]} & {$\#V$} & {$\ell_{slp}$} & {$\ell_{opt}$} \\
		\midrule
		50 & 1e-3 & 5.585884 & 5.585890 & 50 & 12 & 1 & 9 \\
		& 1e-5 & 5.585274 & 5.585280 & 55 & 11 & 1 & 13 \\
		& 1e-7 & 5.585268 & 5.585274 & 64 & 11 & 1 & 5 \\
		\midrule
		100 & 1e-3 & 10.853029 & 10.853039 & 165 & 14 & 5 & 10 \\
		& 1e-5 & 10.852423 & 10.852434 & 268 & 14 & 1 & 5 \\
		& 1e-7 & 10.852417 & 10.852428 & 462 & 14 & 3 & 7 \\
		\midrule
		150 & 1e-3 & 16.593939 & 16.593956 & 239 & 12 & 3 & 9 \\
		& 1e-5 & 16.593333 & 16.593349 & 136 & 13 & 1 & 11 \\
		& 1e-7 & 16.593326 & 16.593343 & 353 & 13 & 1 & 5 \\
		\bottomrule
	\end{tabular}
	\caption{Regularization of sparse families with sparsity density $\rho = 0.25$.}
	\label{tab:reg_supp_3}
\end{table}

Comparing the data in \cref{tab:reg_supp_3} with \cref{tab:5.5}, we deduce that the regularization technique does not yield better performances for a sparsity density of $\rho = 0.25$. This suggests that at this density level, the matrices structure is robust enough to make perturbations essentially useless.

Therefore, for sparsity densities $\rho \ge 0.25$, Algorithm \textbf{(E)} can generally be applied directly without the need for regularization, thus improving computational efficiency.

\subsection{Optimization of the scaling parameter} \label{subsec:opt_theta}

In this section, we discuss the choice of the scaling parameter $\vartheta$ for Algorithm \textbf{(E)}. In particular, we focus on two classes of matrices in relatively high dimensions: \textit{full} and \textit{sparse} matrices.

	\begin{itemize}
		\item For consistency, we initialize the algorithm with the antinorm derived from the leading eigenvalue of $A_1$ for both classes.
		
		\item Since from a theoretical point of view there is no information on optimal values of $\vartheta$, we employ a grid-search strategy in \cref{table:supplement_vartheta_opt} and \cref{table:supplement_vartheta_opt_2}.
	\end{itemize}

\begin{table}[ht!]
	\centering
	\setlength\belowcaptionskip{-8pt}
	\scalebox{0.95}{
		\begin{tabular}{
				@{}
				S[table-format=3.0]
				|
				S[table-format=1.4]
				|
				S[table-format=2.7]
				S[table-format=2.7]
				S[table-format=1.0]
				S[table-format=2.0]
				S[table-format=2.0]
				S[table-format=1.0]
				S[table-format=3.0]
				@{}
			}
			\toprule 
			{\textbf{d}} & {$\vartheta$} & {\textbf{low.bound}} & {\textbf{up.bound}} & {$\ell_{slp}$} & {$\ell_{opt}$} & {$\texttt{n}_{op}$} & {$\texttt{J}_{max}$} & {$\# V$}  \\ \midrule
			50         & 1.0005      & 24.4721163         & 24.4721163        & 1 & 8 & 26 & 3 & 2      \\
			& 1.0010      & 24.4721163         & 24.4721163        & 1 & 8 & 26 & 3 & 2      \\
			& 1.0050      & 24.4721163         & 24.4721163        & 1 & 8 & 26 & 3 & 2      \\
			& 1.0100      & 24.4721163         & 24.4721163        & 1 & 12 & 58 & 4 & 2      \\
			& 1.0200      & 24.4721163         & 24.4721163        & 1 & 12 & 58 & 4 & 2      \\
			& 1.0500      & 24.4721163         & 24.4721163        & 1 & 8 & 26 & 3 & 2      \\
			& 1.1000      & 24.4721163         & 24.4721163        & 1 & 8 & 26 & 3 & 2      \\
			& 1.5000      & 24.4721163         & 24.4721163        & 1 & 8 & 26 & 3 & 3      \\
			& 1.8000      & 24.4721163         & 24.4721163        & 1 & 5 & 20 & 3 & 2      \\ \midrule
			150        & 1.0005      & 75.1124947         & 75.1124947        & 1 & 7 & 16 & 2 & 4      \\
			& 1.0010      & 75.1124947         & 75.1124947        & 1 & 7 & 16 & 2 & 3      \\
			& 1.0050      & 75.1124947         & 75.1124947        & 1 & 7 & 16 & 2 & 3      \\
			& 1.0100      & 75.1124947         & 75.1124947        & 1 & 7 & 20 & 2 & 3      \\
			& 1.0200      & 75.1124947         & 75.1124947        & 1 & 7 & 20 & 2 & 4      \\
			& 1.0500      & 75.1124947         & 75.1124947        & 1 & 7 & 20 & 2 & 4      \\
			& 1.1000      & 75.1124947         & 75.1124947        & 1 & 6 & 18 & 2 & 2      \\
			& 1.5000      & 75.1124947         & 75.1124947        & 1 & 7 & 20 & 2 & 3      \\
			& 1.8000      & 75.1124947         & 75.1124947        & 1 & 7 & 20 & 2 & 3      \\ \bottomrule
	\end{tabular}}
	\caption{Algorithm \textbf{(E)} applied for different values of $\vartheta$ to full matrix families.}
	\label{table:supplement_vartheta_opt}
\end{table}

The behavior for full matrices yields unexpected results, as the value of $\vartheta$ appears to have minimal impact on perfomance. More precisely, referring to \cref{table:supplement_vartheta_opt}, we observe that:
\begin{enumerate}[label=(\alph*)]
	\item In both dimensions tested, the algorithm converges quickly to the target accuracy $\delta=10^{-6}$, regardless of the $\vartheta$ value used.
	
	\item \textbf{($d=50$)} The algorithm correctly identifies the s.l.p. degree as $1$. The matrix $A_2$ is confirmed to be a s.l.p. by the polytopic algorithm, with its extremal antinorm having $4$ vertices. Therefore, the LSR of $\cF$ is:
	\[
	\check \rho(\cF)=\rho(A_2) \approx 24.4721163.
	\]
	The algorithm is extremely efficient, exploring products up to degree $12$ within fewer than $100$ antinorm evaluations ($\texttt{n}_{op}$). Indeed, the low values of $\texttt{J}_{max}$ indicate that each $S_k^\ast$ contains only a few elements - refer to \cref{thm.main} and \cref{prop:improvement} for the precise notations used here.
	
	\item \textbf{($d=100$)} As above, convergence to $\delta=10^{-6}$ is achieved quickly, and the correct s.l.p. degree is identified. This results in the following equality:
	\[
	\check \rho(\cF)=\rho(A_2) \approx 75.1124947.
	\]
	Efficiency is even more significant in this case, with at most $\texttt{n}_{op}=20$ antinorm evaluations required for convergence.
\end{enumerate}

\begin{remark} 
	The convergence appears to be independent of the $\vartheta$ value tested. This unexpected behavior may be attributed to the adaptive procedure incorporating only a few new vertices, reducing the overall impact of $\vartheta$.
\end{remark}

\begin{table}[ht!]
	\centering
	\setlength\belowcaptionskip{-5pt}
	\begin{tabular}{
			@{}
			S[table-format=3.0]
			|
			S[table-format=1.4]
			|
			S[table-format=2.7]
			S[table-format=2.7]
			S[table-format=1.0]
			S[table-format=2.0]
			S[table-format=2.0]
			S[table-format=1.0]
			S[table-format=3.0]
			@{}
		}
		\toprule
		{\textbf{d}} & {$\vartheta$} & {\textbf{low.bound}} & {\textbf{up.bound}} & {$\ell_{slp}$} & {$\ell_{opt}$} & {$\texttt{n}_{op}$} & {$\texttt{J}_{max}$} & {$\# V$} \\ \midrule
		100 & 1.0005 & 12.8611170 & 12.8611170 & 2 & 11 & 66 & 4 & 7 \\
		& 1.0050 & 12.8611170 & 12.8611170 & 2 & 11 & 66 & 4 & 7 \\
		& 1.0150 & 12.8611170 & 12.8611170 & 2 & 11 & 66 & 4 & 6 \\
		& 1.0500 & 12.8611170 & 12.8611170 & 2 & 7  & 32 & 3 & 6 \\
		& 1.5000 & 12.8611170 & 12.8611170 & 2 & 7  & 32 & 3 & 4 \\ \midrule
		150 & 1.0005 & 19.5415187 & 19.5415187 & 1 & 4  & 12 & 2 & 11 \\
		& 1.0050 & 19.5415187 & 19.5415187 & 1 & 4  & 12 & 2 & 11 \\
		& 1.0150 & 19.5415187 & 19.5415187 & 1 & 4  & 12 & 2 & 11 \\
		& 1.0500 & 19.5415187 & 19.5415187 & 1 & 4  & 12 & 2 & 11 \\
		& 1.5000 & 19.5415187 & 19.5415187 & 1 & 4  & 12 & 2 & 11 \\ \bottomrule
	\end{tabular}
	\caption{Algorithm \textbf{(E)} applied to sparse matrices with sparsity density $\rho = 0.3$.}
	\label{table:supplement_vartheta_opt_2}
\end{table}

Analysis of \cref{table:supplement_vartheta_opt_2} shows that Algorithm \textbf{(E)} performs well when applied to sparse matrices with a sparsity density of $\rho = 0.3$. A few conclusive observations:
\begin{enumerate}[label=(\alph*)]
	\item \textbf{($d=100$)} The algorithm converges quickly to the target accuracy $\delta=10^{-6}$ and returns $\ell_{slp}=2$. The candidate s.l.p. identified is $\Pi = A_1A_2$, which is confirmed to be a s.l.p. for $\cF$ by the polytopic algorithm. Hence:
	\[
	\check \rho(\cF) = \rho(A_1A_2)^{1/2} \approx 12.8611169884195,
	\]
	with an extremal antinorm that has $14$ vertices. As in \cref{table:supplement_vartheta_opt}, the algorithm is rather efficient, exploring relatively long products within few antinorm evaluations.
	
	\item \textbf{($d=150$)} The algorithm converges quickly, returning $\ell_{slp}=1$. The polytopic algorithm confirms that $A_1$ is a s.l.p., and hence
	\[
	\check\rho(\cF)=\rho(A_1) \approx 19.5415186589294.
	\]
	The extremal antinorm has only $8$ vertices, fewer than our adaptive ones.
\end{enumerate}

\begin{remark}
	The value of $\vartheta$ appears to have no impact on the result. This is due to no eigenvectors \textit{surviving} the pruning procedure, likely due to the initial antinorm including the leading eigenvector of the s.l.p. $A_1$.
\end{remark}

%%%%%%%%%%%%%%%%%%%%%%%%%%%%%%%%%%%%%%%%%%%%%%%%%%%%%%%%%%%
\subsection{Random families of several matrices}\label{subsections:random}

The results, in the case of families including more than two matrices, are consistent with what we have observed so far: convergence to the pre-fixed accuracy $\delta$ is achieved in a reasonable time and, generally, s.l.p. are identified correctly. A few examples are detailed below:

\smallskip

\begin{enumerate}[label=(\alph*)]
	\item Let $\cF=\{A_1,A_2,A_3\}$, where each $A_i$ is a $50 \times 50$ matrix with all strictly positive entries. Applying Algorithm \textbf{(E)} with $\vartheta = 1.005$, yields:
	\[
	24.5734346483858 \le \check\rho (\cF) \le 24.5734346483858.
	\]
	The process is rather efficient, as suggested by the low computational time of about ten minutes and the performance metric (\cref{remark:alg_1}):
	\[
	\texttt{p} = (\ell_{slp},\ell_{opt},\texttt{n},\texttt{n}_{op},\texttt{J}_{max})= (1,6,6,1092,729).
	\]
	The algorithm identifies a candidate of degree $\ell_{slp}=1$ ($A_1$), which is confirmed to be a s.l.p. with the polytopic algorithm; consequently,
	\[
	\check \rho(\cF)=\rho(A_1) \approx 24.5734346483858.
	\]
	
	\item Let $\cF=\{A_1,A_2,A_3\}$, where each $A_i$ is a $200 \times 200$ matrix with all strictly positive entries. Applying Algorithm \textbf{ (E)} with $\vartheta = 1.005$ leads to
	\[
	100.094951113593 \le \check \rho(\cF) \le 100.095037706193,
	\]
	Therefore, convergence is achieved with absolute accuracy $\delta_{abs}=10^{-4}$ and relative accuracy $\delta=10^{-6}$. The metric performance is
	\[
	\texttt{p}= (\ell_{slp},\ell_{opt},\texttt{n},\texttt{n}_{op},\texttt{J}_{max}) = (1,6,6,1092,729)
	\]
	returning a unique s.l.p. candidate of degree one, $A_3$, that is confirmed to be a s.l.p. with the polytopic algorithm; thus,
	\[
	\check \rho(\cF) = \rho(A_3) \approx 100.095037706193.
	\]
	
	\item Let $\cF=\{A_1,A_2,A_3,A_4\}$ be a family of $150\times 150$ sparse matrices with sparsity density $\rho = 0.5$. Applying Algorithm \textbf{(E)} with $\vartheta=1.005$, yields:
	\[
	29.2101076342739 \le \check \rho(\cF) \le 29.2608364813336.
	\]
	Thus, the algorithm does not converge within the allowed number of antinorm evaluations to $\delta = 10^{-6}$. Nevertheless, the performance metric is
	\[
	\texttt{p}= (\ell_{slp},\ell_{opt},\texttt{n},\texttt{n}_{op},\texttt{J}_{max}) = (1,5,5,1364,1024).
	\]
	and there is a unique s.l.p. candidate of degree one, $\Pi := A_3$. The polytopic algorithm converges in four iterations and confirms that
	\begin{itemize}
		\item $\Pi$ is an actual s.l.p. for $\cF$;
		\item the extremal antinorm consists of $35$ vertices.
	\end{itemize}
\end{enumerate}

%%%%%%%%%%%%%%%%%%%%%%%%%%%%%%%%%%%%%%%%%%%%%%%%%%%%%%%%%%%%%%%%%%%%%%%%%%%%%%%%
\section{Conclusive remarks} \label{sec:conclusion}

In this article, we have extended Gripenberg's algorithm for the first time to approximate the lower spectral radius for a class of families of matrices. This can be efficiently coupled with the polytope algorithm proposed in \cite{GP13} when an exact computation is needed. However, being significantly faster than the polytope algorithm, in applications where an approximation of the lower spectral radius suffices, our algorithm can replace the polytope algorithm.
We have analyzed several versions of the algorithm and in particular, beyond a standard basic formulation, we have considered and extensively experimented two main variations: an adaptive one, utilizing different antinorms, and a second one exploiting the knowledge of eigenvectors of certain products in the matrix semigroup.
Further variants might be successfully explored for specific kind of problems.
The algorithms are publically available and will hopefully be useful to the
community, particularly for their relevance in important applications in approximation theory and stability of dynamical systems.

Finally, we briefly discuss in the next subsection an adaptive Gripenberg's algorithm for the joint spectral radius computation. This extension of the original algorithm uses ideas similar to those proposed for approximating of the lower spectral radius. The obtained results indicate the advantages of using an adaptive strategy.

%%%%%%%%%%%%%%%%%%%%%%%%%%%%%%%%%%%%%%%%%%%%%%%%%%%%%%%%%%%%%%%%%%%%%%%%%%%%%%%%
\subsection{Adaptive Gripenberg's algorithm for the JSR}\label{subsec:adaptive_norm_algorithm}

The adaptive procedure used in Algorithm {\bfseries (A)} suggests that a similar improvement can be implemented in Gripenberg's algorithm \cite{Gripenberg1996}. First, we recall a few definitions from \cite{GP13,GuglielmiZennaro2007,VagnoniZennaro2009}:

\begin{definition} \label{def:absco}
    A set $P \subset \mathbb C^d$ is a \textbf{balanced complex polytope} (b.c.p.) if there exists a minimal set of vertices $V=\{v_1,\ldots,v_p\}$ such that
    \[
    \operatorname{span}(V) = \mathbb C^d \qquad \text{and} \qquad P=\operatorname{absco}(V),
    \]
    where $\operatorname{absco}(\cdot)$ denotes the absolute convex hull. Moreover, a \textbf{complex polytope norm} is any norm $\| \cdot \|_P$ whose unit ball is a b.c.p. $P$.
\end{definition}

Complex polytope norms are dense in the set of all norms on $\mathbb C^d$. This property extends to the induced matrix norms (see \cite{GuglielmiZennaro2007}). Consequently, we have:
\[
\hat \rho(\cF) = \inf_{ \| \cdot \|_P } \max_{A \in \cF} \| A \|_P.
\]
Therefore, even though an extremal complex polytope norm for $\cF$ may not exist, the density property allows an arbitrarily close approximation of the JSR. 

The numerical implementation is similar to Algorithm {\bfseries (A)}, but there are two crucial differences: how to compute $\| \cdot \|_P$ and when to add new vertices. Indeed, if we consider $P=\operatorname{absco}(V)$, then $\|z\|_P$ is given by the solution to
\begin{equation}\label{eq:problem_norm}
\max t_0 \qquad \text{subject to } \begin{cases}
\sum_{x \in V} \alpha_x \text{Re}(x) - \beta_x \text{Im}(x) = t_0  \text{Re}(z)
\\[.4em] \sum_{x \in V} \beta_x \text{Re}(x) + \alpha_x \text{Im}(x) = t_0  \text{Im}(z)
\\[.4em] \sum_{x \in V} \sqrt{\alpha_x^2 + \beta_x^2} \le 1
\end{cases}
\end{equation}
This optimization problem can be solved in the framework of the conic quadratic programming; for more details, refer to \cite{GP13}. The decision to include $z$ in the vertex set relies on the same criterion of Algorithm {\bfseries (A)}, but with a reversed inequality. More precisely
    \begin{itemize}
        \item if $t_0 > 1$, then $z$ falls outside the polytope (so we discard it);
        \item if $t_0 \le 1$, then $z$ is incorporated into the vertex set.
    \end{itemize}

\subsection*{Example of the adaptive Gripenberg algorithm}\label{subsec:il_ex_jsr}

Consider the family:
\[
\cF = \frac15 \left\{  \begin{pmatrix} 3 & 0  \\ 1 & 3 \end{pmatrix}, \begin{pmatrix} 3 & -3  \\ 0 & -1 \end{pmatrix} \right\}.
\]
In \cite{Gripenberg1996}, numerical simulations (on MATLAB) with a maximum number of norm computations $M=52550$ yielded the following bound: $0.6596789 < \hat \rho( \cF) < 0.6596924$.

In this case, the algorithm considers products of length up to $243$ using the $2$-norm. It is also mentioned that increasing the value of $M$ does not improve the bound. %when using the same accuracy $\delta = 10^{-6}$.

Applying the adaptive Gripenberg algorithm with the $1$-norm and setting $M = 250$ yields (in about ten minutes) the following upper bound:
\[
\hat \rho(\cF) \le 0.659678900000002.
\]
This is obtained exploring products of length up to eight, however the optimal gap is achieved at $3$. Additionally, only three vertices are added to the initial $1$-norm. 

\begin{remark}
    The lower bound does not improve as the spectral radius does not benefit from this adaptive procedure.
\end{remark}

\clearpage

	\appendix

	\section{Polytopic antinorms} \label{supplement_sec:comp_antinorm}
	
	Let $\cF$ be a finite family of real non-negative $d \times d$ matrices sharing an invariant cone $K \subseteq \R_+^d$. Consider a polytope antinorm
	\[
	a(\cdot) : K \longrightarrow \R_+,
	\]
	and let $V = \{v_1,\ldots,v_p\}$ be a minimal vertex set for $a(\cdot)$. For any product $P \in \Sigma(\cF)$, we can compute its antinorm using the following formula:
	\begin{equation}\label{eq:supplement_antinorm}
		a(P) := \min_{1 \le i \le p} a(Pv_i).
	\end{equation}
	Thus, being able to calculate $a(z)$ efficiently for any $z$ is fundamental. As mentioned in \cref{subsec:adaptive_antinorm_algorithm}, this can be done by solving the LP problem \eqref{eq:lp_problem}. Since our algorithm relies on the \texttt{linprog} built-in MATLAB function, let us briefly summarize the key points and introduce some relavant notation for \cref{alg.2.1}:
	
	\begin{enumerate}[label=(\alph*)]
		\item \label{enum:a} The function \texttt{linprog} only solves LP problems of the form
		\[
		\min_{x \in \R^N} f^T x \qquad  \text{ subject to } \begin{cases}
			A x  \le b, \\
			A_{eq} x = b_{eq}, \\
			l_b \le x \le u_b.
		\end{cases}
		\]
		Therefore, to write \eqref{eq:lp_problem} in this form, we take $f=(1,0,\ldots,0)^T \in \R^{p+1}$ and we let $x$ be the vector $(c_0,c_1,\ldots,c_p) \in \R^{p+1}$. The inequality constraints
		\[
		c_0 z \ge \sum_{i=1}^p c_i v_i \qquad \text{and} \qquad \sum_{i=1}^p c_i \ge 1
		\]
		must be written as $A x \le b$. Hence, we take $b = (0,\ldots,0,-1)^T \in \R^{d+1}$ and define the augmented matrix:
		\begin{equation}\label{eq:matrix_A}
			A = \begin{pmatrix}
				-z_1 & | & \multicolumn{2}{c}{} & \\ 
				\vdots & | & \multicolumn{2}{c}{\vspace{\jot}V} & \\ 
				-z_d & | & \multicolumn{2}{c}{\vspace{\jot}} & \\ 
				\cline{1-5}
				0 & | & -1 & \ldots & -1
			\end{pmatrix} \in \R^{(d+1),(p+1)}.
		\end{equation}
		The last constraint, $c_i \ge 0$ for each $i$, is recovered by setting $l_b = (0,\ldots,0)^T \in \R^{p+1}$ and leaving $u_b$ empty to indicate that there are no upper bounds on the variables. Notice also that the optimization problem \eqref{eq:lp_problem} has no equality constraints, so we simply let $A_{eq}$ and $b_{eq}$ be empty.
		
		\item To better suit our problem, we customize some options of \texttt{linprog}:
		\begin{itemize}
			\item \texttt{algorithm}: The method used to find the solution. We use the {\itshape interior-point algorithm} since it works well with linear problems.
			\item \texttt{maxIter}: Maximum number of iterations allowed; in our case, we set $\max\{300,d\}$ as a reasonable choice for all dimensions.
			\item \texttt{tol}: Termination tolerance on the dual feasibility.
		\end{itemize}
		%There are other options available (e.g., tolerance on the constraints), but we leave the standard values for simplicity.
		
		\item The implementation requires careful consideration, as it may occasionally fail to return a solution. To address this issue without interrupting the main algorithm (e.g., \cref{alg.1}), we can utilize the \texttt{exitflag} value returned by \texttt{linprog} as it provides information about the reason for termination.
	\end{enumerate}
	
	\begin{algorithm}[ht!]
		\caption{Evaluation of the polytope antinorm $a(\cdot)$ on a vector $z$}
		\label{alg.2.1}
		\begin{algorithmic}[1]
			\STATE{\textbf{Initial data}: $z \in K$ and $V= \begin{pmatrix} v_1 & \cdots & v_p\end{pmatrix}$ vertex matrix corresponding to $a(\cdot)$}
			%\STATE
			\STATE{\textbf{Setting of the LP problem as in \ref{enum:a} above:}}
			\STATE{let $f = (1,0,\ldots,0)^T \in \R^{p+1}$, $b := (0,\ldots,0,-1)^T \in \R^{d+1}$, and denote by $A$ the augmented matrix given in \eqref{eq:matrix_A}}
			\STATE{set $A_{eq}, b_{eq}$ and $u_b$ empty and $l_b = (0,\ldots,0)^T \in \R^{p+1}$}
			\STATE{configure \texttt{linprog}'s options by setting $\texttt{tol}=10^{-10}$ and $\texttt{maxIter} = \max\{300,d\}$ \COMMENT{note that 'interior-point' is the default setting for \texttt{algorithm}}}
			\STATE{let $x$ be the solution obtained via \texttt{linprog} \COMMENT{for simplicity, assume that a solution is found; however, in the actual implementation, the value of \texttt{exitflag} should be checked to handle potential errors}}
			
			\IF{$x(1) = + \infty$}
			\STATE{set $a(z) = 0$}
			\ELSIF{$x(1) = 0$}
			\STATE{set $a(z) = +\infty$}
			\ELSE
			\STATE{set $a(z) = 1/x(1)$}
			\ENDIF		
			\RETURN $a(z)$
		\end{algorithmic}
	\end{algorithm}
	
	We can now exploit the formula \eqref{eq:supplement_antinorm} to compute the antinorm $a(P)$ of any given product matrix $P \in \Sigma(\cF)$. Furthermore, since it is used in the adaptive algorithm, \cref{alg.2.2} returns the candidate vertex $z:=P v_i$, where $v_i \in V$ achieves the minimum in \eqref{eq:supplement_antinorm}.
	
	\begin{remark}
		Algorithm \textbf{(A)}, detailed in \cref{supplement_sec:adaptive_alg}, incorporates new vertices immediately if a criterion is satisfied. As a result, \cref{alg.2.2} is applied to each product $P \in S_k$ with a potentially different polytope antinorm.
	\end{remark}
	
	\begin{algorithm}[ht!]
		\caption{Polytope antinorm evaluation at a matrix product $P \in \Sigma(\cF)$}
		\label{alg.2.2}
		\begin{algorithmic}[1]
			\STATE{\textbf{Initial data}: $P \in \Sigma(\cF)$ and $V= \begin{pmatrix} v_1 & \cdots & v_p\end{pmatrix}$ vertex matrix of $a(\cdot)$}
			\FOR{$i = 1$ to $p$}
			\STATE{compute the product $z_i = A \cdot v_i$ and its antinorm $a(z_i)$ with \cref{alg.2.1}}
			\STATE{store the value in an auxiliary vector $c(i) = a(z_i)$}
			\ENDFOR
			\STATE{set $a(P) = \min_{1 \le i \le p} c(i)$ and let $j$ be the index such that $a(P)=c(j)$.}
			\RETURN $a(P)$ and $z := P \cdot v_j$ \COMMENT{i.e., the potential new vertex}
		\end{algorithmic}
	\end{algorithm}

	%%%%%%%%%%%%%%%%%%%%%%%%%%%%%%%%%%%%%%%%%%%%%%%%%%%%%%%%%%%%%%%%%%
	\subsection{Vertices pruning} \label{supplementsec:pruning_vertices}
	
	As mentioned in \cref{subsec:adaptive_antinorm_algorithm}, at the end of each step of the adaptive algorithms, we have an augmented vertex set $V_{aux}$ that is likely not minimal. To address this issue, we employ a pruning procedure before the next algorithm step. However, in \cref{alg.2.2.1}, we need to be careful and avoid:
	\begin{enumerate}[label=(\roman*)]
		\item the accidental removal of non-redundant vertices;
		\item the same vertex appearing multiple times.
	\end{enumerate}
	
	\begin{algorithm}[ht!]
		\caption{Vertices pruning procedure}
		\label{alg.2.2.1}
		\begin{algorithmic}[1]
			\STATE{\textbf{Initial data}: $V= \begin{pmatrix} v_1 & \cdots & v_p\end{pmatrix}$ vertex matrix, $r=0$  \COMMENT{to determine when to stop the procedure} and $\texttt{tol}>0$ \COMMENT{used in the elimination criterion}}
			%	\STATE
			\WHILE{$r$ is equal to $0$}
			\STATE{let $n_v$ be the current number of vertices (i.e, columns) of $V$}
			\IF{$n_v$ is less than or equal to $1$}
			\STATE{exit the procedure}
			\ENDIF
			\STATE{set $r=1$ \COMMENT{this values is set to zero in the \textbf{for} cycle below if at least one vertex is eliminated; otherwise, it remains one and the \textbf{while} does not restart}}
			\FOR{$i = n_v, n_v-1, \ldots, 1$}
			\STATE{let $v_i$ denote the $i$-th column of $V$ and define $W_i$ as the matrix obtained by removing the $i$-th column from $V$}
			\IF{$\operatorname{rank}(W)$ is equal to $1$}
			\STATE{skip directly to the next index ($i \to i-1$)}
			\ENDIF
			\STATE{compute $a_{W_i}(v_i)$, where $a_{W_i}$ is the polytope antinorm with vertex set $W_i$}
			\IF{$a_{W_i}(v_i) \ge 1 + \texttt{tol}$}
			\STATE{set $r=0$ and remove the column $v_i$ from $V$, i.e. let $V = W_i$}
			\ENDIF
			\ENDFOR
			\ENDWHILE
			\STATE{identify all columns of $V$ that appear more than once and remove them}
			%	\STATE
			\RETURN $V$
		\end{algorithmic}
	\end{algorithm}
	
	In particular, we examine each $v_i$ in the vertex set $V$ by considering the polytope antinorm obtained by excluding $v_i$ (which we denote by $a_{W_i}$) and, ideally, whenever $v_i$ lies on the boundary of said polytope, removing it. However, to avoid removing non-redundant vertices due to numerical inaccuracies, we fix a small tolerance (\texttt{tol}) and use the following inequality as the exclusion criterion:
	\[
	a_{W_i}(v_i) \ge 1 + \texttt{tol}.
	\]
	Finally, we make a few remarks about \cref{alg.2.2.1}:
	\begin{itemize}
		\item The stopping criterion $r=1$ halts the procedure as soon as an entire \textbf{for} cycle does not remove any vertex.
		
		\item The procedure may be computationally expensive, as a different antinorm $a_{W_i}(\cdot)$ appears at each iteartion of the \textbf{for} cycle.
		
		\item The fixed tolerance \texttt{tol} should be chosen to be small to maximize the removal of redundant vertices. Setting $\texttt{tol} = 0$, while possible, is not recommended, as numerical inaccuracies may lead to the removal of important vertices.
	\end{itemize}

	%%%%%%%%%%%%%%%%%%%%%%%%%%%%%%%%%%%%%%%%%%%%%%%%%%%%%%%%%%%%%%%%%%%%%%%%
	\subsection{Identification of s.l.p. candidates} \label{supplementsec:identification_slp}
	
	Recall that a product $\Sigma \in \Sigma(\cF)$ of degree $k$ is a s.l.p. for $\cF$ if the following holds:
	\[
	\rho(\Sigma)^{1/k} = \inf_{k \ge 1} \min_{P \in \Sigma_k(\cF)} \rho(P)^{1/k}
	\]
	Referring to the notation of \cref{rmk:slp}, it suffices to calculate $\ell_{slp}$ and then select among products of such degree those minimizing $\rho(P)^{1/\ell_{slp}}$. The numerical implementation is given in \cref{alg.3}, but first, we summarize the strategy:
	\begin{enumerate}[label=(\arabic*)]
		\item Set $\ell_{slp}:=1$. Next, enter the main \textbf{while} loop of \cref{alg.1} (or any adaptive variant) and complete the step.
		
		\item If the new upper bound is \textbf{strictly} less than the previous value, set $\ell_{slp}$ equal to the current degree. If not, $\ell_{slp}$ remains unchanged.
		
		\item When the algorithm terminates (due to convergence or reaching the maximum number of antinorm evaluations), the value $\ell_{slp}$ coincides with the one given in \cref{rmk:slp}.
		
		\item Solve the minimization problem
		\[
		\min_{P \in \Sigma_{\ell_{slp}(\cF)}} \rho(P)^{1/\ell_{slp}},
		\]
		and apply the polytopic algorithm \cite{GP13} to all solutions to verify if they are s.l.p.
		%If $\ell_{slp}$ is large, it may be difficult to solve the problem in $\Sigma_{\ell_{slp}}$; thus, consider
		%\[
		%\min_{P \in S_{\ell_{slp}(\cF)}^\ast} \rho(P)^{1/\ell_{slp}}
		%\]
		%since $S_{\ell_{slp}(\cF)}^\ast$ is typically smaller due to the efficiency of our algorithms.
		
	\end{enumerate}
	
	This procedure can be integrated into any version of the algorithm, but performs better with the adaptive variants. Indeed, a fixed antinorm may lead \cref{alg.1} to evaluate a large number of products at each degree and, consequently, limit its ability to explore higher degrees within a reasonable time-frame. In contrast, the adaptive variants consider sets $S_k^\ast$ with typically smaller cardinality and, hence, take into account higher-degree products at the same computational cost. 
	
	There is a potential drawback to this approach: adaptive algorithms might inadvertently remove actual s.l.p.s, resulting in the wrong value for $\ell_{slp}$. Despite this concern, numerical simulations suggest that this is not a common issue. 
	
	\begin{remark}
		A concrete example comparing the performance of \cref{alg.1} and \textbf{(A)} and \textbf{(E)} in s.l.p. identification can be found in \cref{sec:il_ex}.
	\end{remark}
	
	The numerical implementation of this procedure requires only a few modifications to our algorithms. As such, we will only discuss the lines that should be changed or added, and refer to \cref{alg.1} or \cref{alg.supplement.a} for the complete picture.
	
	\begin{algorithm}[ht]
		\caption{Identification of s.l.p. candidates}
		\label{alg.3}
		\begin{algorithmic}[1]
			\STATE{\textbf{Let $\ell_{slp}$ be the value returned by either \cref{alg.1} or \cref{alg.supplement.a} and do the following:}}
			
			\IF{$\ell_{slp}$ is small enough}
			\STATE{solve the minimization problem $\min_{P \in \Sigma_{\ell_{slp}}(\cF)} \rho(P)^{1/\ell_{slp}}$}
			\ELSIF{$\ell_{slp}$ is large}
			\STATE{solve the minimization problem $\min_{P \in S_{\ell_{slp}}^\ast} \rho(P)^{1/\ell_{slp}}$ \COMMENT{the definition of $S_{\ell_{slp}}^\ast$ is given in \cref{prop:improvement}}}
			\ENDIF
			
			\RETURN all solutions (up to cyclic permutations) of the minimization problem
		\end{algorithmic}
	\end{algorithm}

	%%%%%%%%%%%%%%%%%%%%%%%%%%%%%%%%%%%%%%%%%%%%%%%%%%%%%%%%%%%%%%%%%%%%%%
	
	\section{Implementation of Algorithm (A)/(E)} \label{supplement_sec:adaptive_alg}
	
	We now present the numerical implementations of the adaptive variants of \cref{alg.1}: Algorithms \textbf{(A)} and Algorithm \textbf{(E)}. For the latter, we only highlight the main differences and omit the complete pseudocode due to their similarity.
	
	Let us examine the complete pseudocode of Algorithm \textbf{(A)}. We will maintain the notation from \cref{alg.1} whenever possible for clarity. 
	
	\begin{remark}
		The polytope antinorm is defined by its vertex set, which is refined throughout the algorithm. Thus, we use the notation $a_V(\cdot)$ to refer to the antinorm corresponding to the current vertex set $V$.
	\end{remark}
	
	\clearpage
	
	\begin{algorithm}[ht!]
		\caption{Adaptive Algorithm \textbf{(A)}}
		\label{alg.supplement.a}
		\begin{algorithmic}[1]
			\STATE{Set the initial bounds for the lower spectral radius: $L=0$ and $H=+\infty$}
			\STATE{Let $S_1 := \cF$ \COMMENT{to store matrices of degree one, following the notation of \cref{thm.main}} and set $m = \# S_1$ \COMMENT{number of elements in the family $\cF$}}
			%\STATE{\textbf{First step of the algorithm:}}
			\FOR{$i=1$ to $m$}
			\STATE{compute $a_V(A_i)$ with \cref{alg.2.2} and let $z_i$ be the candidate vertex}
			\STATE{set $l(i) = a_V(A_i)$ \COMMENT{the vector $l$ stores all candidates lower bounds in the current step} and $H = \min\{H, \rho(A_i)\}$ \COMMENT{update the upper bound}}
			\IF{$a_V(z_i) \le 1 + \texttt{tol}$}
			\STATE{add $z_i$ to the vertex set $V$, i.e. set $V=[V, z_i]$}
			\ENDIF
			\ENDFOR
			\STATE{compute the new lower bound as $L = \min_{1 \le i \le m} l(i)$ and prune the vertex set $V$ by applying \cref{alg.2.2.1}}
			% \STATE{\textbf{Setting of iterations parameters and main loop:}}
			\STATE{set $\texttt{n} = 1$ \COMMENT{current degree}, $\texttt{n}_{op} = m$ \COMMENT{number of antinorm evaluations}, $\texttt{J} = m$ \COMMENT{cardinality of $S_{\texttt{n}}$, coincides with $m$ after the first step}, and $\texttt{J}_{max} = \texttt{J}$ \COMMENT{keeps track of the maximum value of \texttt{J}}}
			\STATE{set $\ell_{opt}=1$ \COMMENT{i.e., the degree for which the gap between lower and upper bounds is optimal} and $\ell_{slp}=1$ \COMMENT{i.e., degree yielding the optimal upper bound}}
			\WHILE{$H - L \ge \delta \quad \&\& \quad \texttt{n}_{op} \le M$}
			\STATE{set $H_{old} = H$, $L_{old} = L$ and $\texttt{J}_{new} = 0$}
			\STATE{increase $\texttt{n}=\texttt{n}+1$ and initialize $S_{\texttt{n}} = [\cdot]$ \COMMENT{empty set to store matrix products of degree \texttt{n} for the next iteration, following the notation of \cref{thm.main}}}
			\FOR{$k = 1$ to $\texttt{J}$}
			\FOR{$i = 1$ to $m$}
			\STATE{let $X_k$ be the $k$-th element in $S_{\texttt{n}-1}$ \COMMENT{note: $S_{\texttt{n}-1}$ has cardinality \texttt{J}}}
			\STATE{set $q = \texttt{J}_{new} + 1$ and consider the product $Y := X_k A_i$}
			\STATE{compute $a_V(Y)$ and let $z_Y$ be the candidate vertex}
			\STATE{set $l_{new}(q) = \max\{ l(k), (a_V(Y))^{1/\texttt{n}}\}$ and $H := \min\{ H, (\rho(Y))^{1/\texttt{n}} \}$}
			\IF{$a_V(z_Y) \le 1 + \texttt{tol}$}
			\STATE{add $z_Y$ to the vertex set $V$, i.e. set $V=[V, z_Y]$}
			\ENDIF
			\IF{$l_{new}(q) < H - \delta$}
			\STATE{increase $\texttt{J}_{new} = \texttt{J}_{new} + 1$ and compute $L = \min\{L, l_{new}(q)\}$}
			\STATE{include $Y$ in the set $S_{\texttt{n}}$ \COMMENT{which will be used in the next iteration}}
			\ENDIF
			\ENDFOR
			\ENDFOR
			%\STATE{\textbf{Updates to prepare for next iteration:}}
			\STATE{compute $L = \max\{ L_{old}, \min\{ L, H - \delta \}\}$ and set $l := l_{new}$}
			\STATE{set $\texttt{n}_{op} = \texttt{n}_{op} + \texttt{J} \cdot m$, $\texttt{J} = \texttt{J}_{new}$ and $\texttt{J}_{max} = \max\{\texttt{J}, \texttt{J}_{max} \}$}
			\IF{$H - L < H_{old} - L_{old}$}
			\STATE{set $\ell_{opt} := \texttt{n}$ \COMMENT{store the optimal degree so far (for convergence)}}
			\ENDIF
			\IF{$H<H_{old}$}
			\STATE{set $\ell_{slp} := \texttt{n}$ \COMMENT{store the degree yielding the best upper bound so far}}
			\ENDIF
			\STATE{prune the current vertex set $V$ by applying \cref{alg.2.2.1}}
			\ENDWHILE
			
			\STATE \textbf{return} $\texttt{lsr}:=(L, H)$,
			$\texttt{p} = (\ell_{opt}, \ell_{slp}, \texttt{n}, \texttt{n}_{op}, \texttt{J}_{max})$ and $V$ \COMMENT{final vertex set}
		\end{algorithmic}
	\end{algorithm}
	
		\clearpage
	Algorithm \textbf{(E)}, on the other hand, has an additional input value: the scaling parameter $\vartheta$. The main changes to \cref{alg.supplement.a} are the following:
	\begin{itemize}
		\item Modify the first \textbf{for} cycle (lines 3--9). Specifically, store all matrices $A_j$ such that in line 5 we have
		\[
		\min\{H,\rho(A_j)\} = \rho(A_j).
		\] 
		After all elements in $\cF$ have been evaluated, before the pruning procedure starts (line 10), find the leading eigenvector $v_j$ of each $A_j$ stored previously and add to the vertex set $V$ the rescaling
		\begin{equation}\label{eq:riscaled}
			\widetilde v_j := v_j \left( a_V(v_j) \vartheta \right)^{-1}.
		\end{equation}
		
		\item The same strategy is employed in the main \textbf{while} loop. Specifically, store all matrix products such that
		\[
		\min\{H, (\rho(Y))^{1/\texttt{n}}\} = (\rho(Y))^{1/\texttt{n}},
		\] 
		and, before the pruning procedure in line $39$, add to $V$ all the leading eigenvectors rescaled as in \eqref{eq:riscaled}.
	\end{itemize}
	
	\begin{remark}
		An interesting improvement to this algorithm would be to develop a method for adaptively change the value of $\vartheta$ after each step. The idea is to find a compromise between two conflicting requirements: preventing drastic changes to the antinorm (i.e., $\vartheta$ not too large) and ensuring that the antinorm actually improves (i.e., $\vartheta>1$).
	\end{remark}

	%%%%%%%%%%%%%%%%%%%%%%%%%%%%%%%%%%%%%%%%%%%%%%%%%%%%%%%%%%%%%%%%%%%%%%
	\section{Output of Algorithm (A)} \label{supplement:simulation}
	
	We now describe in details the first two steps of Algorithm \textbf{(A)} applied to the example in \cref{subsec:il_ex}. Specifically, the input consists of the matrix family already rescaled, i.e.
	\[
	\widetilde\cF := \rho(\Pi)^{-1/8} \cdot \left\{ \begin{pmatrix}
		7 & 0 \\ 2 & 3
	\end{pmatrix}, \begin{pmatrix}
		2 & 4 \\ 0 & 8
	\end{pmatrix} \right\},
	\]
	where $\Pi := A_1A_2(A_1^2A_2)^2$. Fix the accuracy to $\delta=10^{-6}$ and the initial antinorm to $a_1(\cdot)$ (the $1$-antinorm), which corresponds to the vertex matrix
	\[
	V_{in}=\begin{pmatrix}
		1 & 0 \\ 0 & 1
	\end{pmatrix}.
	\]
	
	\subsection*{Step 1: degree-one products (lines 3--9 of \cref{alg.supplement.a})}
	
	The first step only considers the rescaled matrices $\tilde A_1$ and $\tilde A_2$ to compute the bounds:
	\begin{enumerate}[label=(\arabic*)]
		\item The matrix $\tilde A_1$, when rounded to four decimal digits, is given by
		\[
		\tilde A_1 \approx \begin{pmatrix}
			1.1649  &   0.3328 \\
			0  &  0.4992
		\end{pmatrix}.
		\]
		Its $1$-antinorm is $a_1(\tilde A_1) = 0.8320$, and the candidate vertex returned by \cref{alg.2.2} is $z = (0.3328,0.4992)^T$. Recall that $z$ is characterized by
		\[
		a_1(A_1)=a_1(A_1 z).
		\]
		Set $l(1) := a_1(\tilde A_1)$ (refer to line 6), compute the corresponding spectral radius $\rho(\tilde A_1) = 1.1649$ and update the upper bound accordingly:
		\[
		H=\min\{H,\rho(\tilde A_1)\}=\min\{+\infty,1.1649\}=1.1649.
		\]
		Store the matrix $\tilde A_1$ for the next step by setting $S_1=\{\tilde{A_1}\}$, and go to line 7. Since $a(z) = 0.8320$, we add $z$ to the current vertex set:
		\[
		V= \begin{pmatrix}V_{in} & z \end{pmatrix} = \begin{pmatrix}
			1 & 0 & 0.3328 \\ 0 & 1 & 0.4992
		\end{pmatrix}.
		\]
		
		\item The matrix $\tilde A_2$, rounded to four decimal places, is given by
		\[
		\tilde A_2 \approx \begin{pmatrix}
			0.3328    &     0 \\
			0.6656 &   1.3313
		\end{pmatrix}.
		\]
		Calculate its antinorm with respect to the new vertex set, $a(\tilde A_2) = 1.0528$, obtaining the new candidate vertex $z = (0.1108, 0.8861)^T$. Store its value by setting $l(2)=a(\tilde A_2)$, compute the spectral radius $\rho(\tilde A_2)= 1.3313$, and update the upper bound:
		\[
		H=\min\{H,\rho(\tilde A_2)\} = \min\{1.1649,1.3313\} = 1.1649,
		\]
		which, in this case, does not change. Next, store the matrix $\tilde A_2$ as follows:
		\[
		S_1= S_1 \cup \{\tilde A_2\} = \{\tilde A_1,\tilde A_2\},
		\]
		and proceed to line 7. The antinorm of $z$ is $a(z) = 1.0528$, so we do not add $z$ to the vertex set, which remains as before:
		\[
		V=\begin{pmatrix}
			1 & 0 & 0.3328 \\ 0 & 1 & 0.4992
		\end{pmatrix}.
		\]
	\end{enumerate}

	\subsection*{Pruning, lower bound and iterative parameters}
	
	In line 11, compute the lower bound after step one as $\min_i l(i)$, which in this case yields:
	\[
	L=\min\{l(1),l(2)\}=l(1)=0.8320.
	\]
	The pruning procedure applied to $V$ does not remove any vertex. Next, proceed to lines 13--14 and initialize the iteration parameters:
	\[
	\texttt{n}=1, \quad \texttt{n}_{op}=2, \quad \texttt{J}=\texttt{J}_{max}=2.
	\]
	Also set $\ell_{opt}=\ell_{slp}=1$ to keep track of the degree yielding the optimal gap and the best (smallest) upper bound respectively.
	
	\subsection*{Step $2$: degree-two products (lines 13--39)}
	
	Increase $\texttt{n}$ to $2$ and start the two nested \textbf{for} cycles:
	
	\begin{enumerate}[label=(\arabic*)]
		\item[(1a)] When $k=1$ and $i=1$, consider $\tilde A_1 \in S_1$ and set $Y=: \tilde A_1^2$, which rounded to the fourth decimal digit is given by
		\[
		Y \approx \begin{pmatrix}
			1.3569   & 0.5538\\
			0  &  0.2492
		\end{pmatrix}.
		\]
		Following line 20, compute the antinorm $a(Y) = 0.8869$ to obtain the candidate vertex $z = (0.5538,0.2492)^T$ and, then, set
		\[
		l_{new}(1)=\max\{l(1),a(Y)^{1/2}\} = a(Y)^{1/2} = 0.9418.
		\]
		Compute the spectral radius $\rho(Y) = .3569$, and note that the upper bound cannot improve since $\rho(A^2)=\rho(A)^2$; therefore,
		\[
		H=\min\{H,\rho(Y)^{1/2}\} = H = \rho(Y)^{1/2} = 1.1649.
		\]
		Next (lines 22--24), establish whether or not $z$ is added to the vertex set. In particular, since $a(z) = 0.8869$ is less than $1$, it is incorporated:
		\[
		V=\begin{pmatrix}
			1 & 0 & 0.3328 & 0.5538 \\ 0 & 1 & 0.4992 & 0.2492
		\end{pmatrix}.
		\]
		Finally (line 25), check the condition $l_{new}(1)<H-\delta$. Since it is satisfied, update the lower bound
		\[
		L=\min\{L,l_{new}(1)\}=l_{new}(1)=0.9418,
		\]
		and store $Y$ for the next step by setting $S_2=\{\tilde A_1^2\}$.
		
		\item[(1b)] When $k=1$ and $i=2$, set $Y := \tilde A_1 \tilde A_2$, which rounded to the fourth decimal digit is given by
		\[
		Y\approx \begin{pmatrix}
			0.6092  &  0.4431 \\
			0.3323 &   0.6646
		\end{pmatrix}.
		\]
		The antinorm is $a(Y) = 0.9778$ and the corresponding candidate vertex is $z = (0.4478,0.3497)^T$. Set, once again, the quantity
		\[
		l_{new}(2)=\max\{l(1),a(Y)^{1/2}\}  = a(Y)^{1/2} = 0.9888.
		\]
		The spectral radius of $Y$ is $\rho(Y) = 1.0216$ so, since there is an improvement, update the upper bound as follows:
		\[
		H=\min\{H,\rho(Y)^{1/2}\} = \rho(Y)^{1/2} = 1.0108.
		\]
		Next, establish whether or not $z$ is added to the vertex set. Since $a(z) = 0.9778$, the matrix $V$ is updated once again:
		\[
		V=\begin{pmatrix}
			1 & 0 & 0.3328 & 0.5538 & 0.4478 \\ 0 & 1 & 0.4992 & 0.2492 & 0.3497
		\end{pmatrix}.
		\]
		Finally, the \textbf{if} condition in line 25 is satisfied; hence, compute the new lower bound as
		\[
		L=\min\{L,l_{new}(2)\}=l_{new}(2)=0.9418,
		\]
		and store $Y$ in $S_2$ by setting $S_2=\{\tilde A_1^2,\tilde A_1 \tilde A_2\}$.
		
		\item[(2a)] When $k=2$ and $i=1$, consider $\tilde A_2 \in S_1$ and set $Y=\tilde A_2 \tilde A_1$, which rounded to the fourth decimal digit is given by
		\[
		Y \approx \begin{pmatrix}
			0.3877 &   0.1108 \\
			0.7754  &  0.8861
		\end{pmatrix}.
		\]
		The antinorm of $Y$ is $a(Y) = 0.9766$, while the corresponding candidate vertex is $z = (0.2123,0.6571)^T$. Set the quantity:
		\[
		l_{new}(3)=\max\{l(2),a(Y)^{1/2}\} = l(2)=1.0528.
		\]
		Calculate the spectral radius $\rho(Y) = 1.0216$ and notice that the upper bound does not change since $\rho(AB)=\rho(BA)$:
		\[
		H=\min\{H,\rho(Y)^{1/2}\} = H = \rho(Y)^{1/2} = 1.0108.
		\]
		Next establish whether or not $z$ is added to the vertex set. Since $a(z) = 0.9766$ is less than $1$, the matrix $V$ is updated:
		\[
		V=\begin{pmatrix}
			1 & 0 & 0.3328 & 0.5538 & 0.4478 & 0.2123 \\ 0 & 1 & 0.4992 & 0.2492 & 0.3497 & 0.6571
		\end{pmatrix}.
		\]
		However, the condition in line 25 is not satisfied ($l_{new}(3)>H-\delta$), which means that the product $Y$ is discarded and both the lower bound $L$ and the set $S_2$ do not change.
		
		\item[(2b)] When $k=2$ and $i=2$, set $Y := \tilde A_2^2$, which rounded to the fourth decimal digit is given by
		\[
		Y \approx \begin{pmatrix}
			0.1108  &      0 \\
			1.1077  &  1.7723
		\end{pmatrix}.
		\]
		Compute $a(Y) = 1.1542$, the candidate vertex $z = (0.0613,1.0552)^T$, and set the quantity:
		\[
		l_{new}(4)=\max\{l(2),a(Y)^{1/2}\} =a(Y)^{1/2} = 1.0743.
		\]
		The spectral radius is $\rho(Y)=1.7723$, so it does not improve the upper bound:
		\[
		H=\min\{H,\rho(Y)^{1/2}\} = H = 1.0108.
		\]
		Next, since $a(z) = 1.1542$, the candidate vertex is discarded. Furthermore, the condition in line 25 is not satisfied ($l_{new}(4)>H-\delta$), so $Y$ is discarded and both $L$ and $S_2$ remain unchanged.
	\end{enumerate}
	
	\subsection*{Update of iteration parameters and conclusion}
	
	Only two products of degree two survived this step since, following the notation of \cref{thm.main}, we have
	\[
	S_2=\{\tilde A_1^2,\tilde A_1 \tilde A_2 \}.
	\]
	This set is used to generate products of degree three, meaning that $S_3$ may have at most $4$ elements instead of $3^2=9$. Next, update the lower bound (line 31) as
	\[\begin{aligned}
		L & = \min\{L_{old},\min\{L,H-\delta\}\} 
		\\ & = \min\{0.8320, \min\{0.9418,  1.0108-10^{-6}\}\} = 0.9418.
	\end{aligned}\]
	Update the iteration parameters: $\texttt{n}_{op}=6$, $\texttt{J}=2=\texttt{J}_{max}$ and, finally, since both the gap and the upper bound improved, we set
	\[
	\ell_{slp}=\ell_{opt}=\texttt{n}=2.
	\]
	The vertex pruning (line 39) at the end of the step yields
	\[
	V=\begin{pmatrix}
		1 & 0 & 0.3328 & 0.5538 & 0.4478 & 0.2123 \\ 0 & 1 & 0.4992 & 0.2492 & 0.3497 & 0.6571
	\end{pmatrix},
	\]
	which means that once again no vertex is removed during the procedure.

\SkipTocEntry\section*{Acknowledgments}

The code used to obtain the data in this paper is provided in a repository at \url{https://github.com/francescomaiale/subradius_computation}

Nicola Guglielmi and Francesco Paolo Maiale acknowledge that their research was supported by funds from the Italian 
MUR (Ministero dell'Universit\`a e della Ricerca) within the 
PRIN 2022 Project ``Advanced numerical methods for time dependent parametric partial differential equations with applications''. Nicola Guglielmi acknowledges the support of the Pro3 joint project entitled
``Calcolo scientifico per le scienze naturali, sociali e applicazioni: sviluppo metodologico e tecnologico''.
He is also affiliated to the INdAM-GNCS (Gruppo Nazionale di Calcolo Scientifico).

\bibliographystyle{siam} 
\bibliography{references.bib}

\end{document}